\documentclass[12pt,a4paper]{amsart}
\usepackage[english]{babel}

\newtheorem{theorem}{Theorem}[section]

\newtheorem{lemma}[theorem]{Lemma}
\newtheorem{proposition}[theorem]{Proposition}

\newtheorem*{claim*}{Claim}
\theoremstyle{definition}
\newtheorem{definition}{Definition}
\newtheorem{construction}{Construction}

\newtheorem{remark}[theorem]{Remark}

\newcommand{\purple}[1]{\textcolor{black}{#1} }

\usepackage[square,sort,comma,numbers]{natbib}

\usepackage[all]{xy}
\usepackage{pstricks}
\usepackage{enumerate}
\usepackage{amsfonts,amssymb,amsmath,eucal,pinlabel,array,hhline}
\usepackage{slashed}
\usepackage{tabulary}
\usepackage{fancyhdr}
\usepackage{color}
\usepackage{a4wide}
\usepackage{bbm}
\usepackage[position=b]{subcaption}
\usepackage[colorlinks,
    linkcolor={blue!50!black},
    citecolor={blue!50!black},
    urlcolor={black!80!black}]{hyperref}

\newcommand{\Z}{\mathbb{Z}}
\newcommand{\Q}{\mathbb{Q}}

\newcommand{\bsm}{\left(\begin{smallmatrix}}
\newcommand{\esm}{\end{smallmatrix}\right)}
\newcommand{\id}{\operatorname{Id}}
\newcommand{\Bl}{\operatorname{Bl}}
\newcommand{\coker}{\operatorname{coker}}

\newcommand{\Homeo}{\operatorname{Homeo}}
\newcommand{\Aut}{\operatorname{Aut}}
\newcommand{\hAut}{\operatorname{hAut}}

\newcommand{\PD}{\operatorname{PD}}
\newcommand{\ev}{\operatorname{ev}}
\newcommand{\quadr}{\operatorname{q}}

\newcommand{\Hom}{\operatorname{Hom}}

\newcommand{\ord}{\operatorname{ord}}

\newcommand{\ks}{\operatorname{ks}}
\newcommand{\st}{\operatorname{st}}
\newcommand{\sm}{\setminus}
\newcommand{\ol}{\overline}
\newcommand{\wt}{\widetilde}
\DeclareMathOperator{\spin}{spin}
\newcommand{\smfrac}[2]{\mbox{\footnotesize$\displaystyle\frac{#1}{#2}$}} 

\DeclareSymbolFont{EulerScript}{U}{eus}{m}{n}
\DeclareSymbolFontAlphabet\mathscr{EulerScript}

\begin{document}

\title{Infinite homotopy stable class for 4-manifolds with boundary}

\author[A.~Conway]{Anthony Conway}
\address{Massachusetts Institute of Technology, Cambridge MA 02139, United States}
\email{anthonyyconway@gmail.com}
\author[D.~Crowley]{Diarmuid Crowley}
\address{School of Mathematics and Statistics, University of Melbourne, Parkville,
\newline \indent VIC, 3010, Australia}
\email{dcrowley@unimelb.edu.au}
\author[M.~Powell]{Mark Powell}
\address{School of Mathematics and Statistics, University of Glasgow, United Kingdom}
\email{mark.powell@glasgow.ac.uk}

\def\subjclassname{\textup{2020} Mathematics Subject Classification}
\subjclass{Primary~57K40, 57R65. 
}

\begin{abstract}
We show that for every odd prime $q$, there exists an infinite family~$\{M_i\}_{i=1}^{\infty}$ of topological 4-manifolds that are all stably homeomorphic to one another, all the manifolds $M_i$ have isometric rank one equivariant intersection pairings
and boundary $L(2q, 1) \# (S^1 \times S^2)$, but they are pairwise not homotopy equivalent via any homotopy equivalence that restricts to a homotopy equivalence of the boundary.
\end{abstract}
\maketitle


\section{Introduction}
In what follows a manifold is understood to mean a compact, connected, oriented, topological manifold.
Let~$W_g := \#_g (S^2 \times S^2)$ be the~$g$-fold connected sum of~$S^2 \times S^2$ with itself.
Two $4$-manifolds~$M$ and~$N$ with the same Euler characteristic are \emph{stably homeomorphic},  denoted~$M \cong_{\st} N$,  if there exists a nonnegative integer~$g$ and a homeomorphism
\[M \# W_g \cong N \# W_g.\]

Surgery theory suggests two ways to classify 4-manifolds. The classical Browder-Novikov-Sullivan-Wall~\cite{WallSurgery} approach is to classify up to homotopy equivalence and then employ the surgery exact sequence. Kreck's modified surgery approach~\cite{KreckSurgeryAndDuality} seeks to classify up to stable homeomorphism, and then attempt to destabilise.  A natural question then arising is to compare the homotopy and stable classifications.
To do this precisely for 4-manifolds with boundary we fix a 4-manifold $M$ and define the \emph{homotopy stable class}:
\[ \mathcal{S}^{\st}_h(M) := \{N  \mid N \cong_{\st} M \}/\text{homotopy equivalence of pairs}.\]
Here,  we understand a \emph{homotopy equivalence of pairs}~$N_1 \simeq N_2$ to be one that restricts to a homotopy equivalence between the boundaries.
When the manifolds are closed, this recovers the usual notion of homotopy equivalence.

Using the equivariant intersection form~$\lambda_N$ of~$N$ as an invariant, ~$\mathcal{S}^{\st}_h(M)$ can be arbitrarily large: for example,  one can use Freedman's work to realise distinct positive definite symmetric bilinear forms with the same signature and rank by simply-connected closed~$4$-manifolds with identical Kirby-Siebenmann invariant~\cite{Freedman}.
For this reason,  we study the homotopy stable class one intersection form at a time and set
$$ \mathcal{S}^{\st}_{h,\lambda}(M) := \{N  \mid N \cong_{\st} M, \lambda_N \cong \lambda_M \}/\text{homotopy equivalence of pairs}.$$
If~$M$ is closed and has~$\pi_1(M)=1$, $\Z$, or $\Z/n$, or $\pi_1(M)$ is a solvable Baumslag-Solitar group, then~$|\mathcal{S}^{\st}_{h,\lambda}(M)|=1$: stably homeomorphic manifolds with isometric equivariant intersection forms are homeomorphic by work of Freedman for~$\pi_1=1$~\cite{Freedman}, Freedman-Quinn for~$\pi_1 \cong \Z$~\cite{FreedmanQuinn}, Hambleton-Kreck for $\pi_1 \cong \Z/n$~\cite[Theorem~C]{HK93}, and Hambleton-Kreck-Teichner for solvable Baumslag-Solitar group~\cite[Theorem~A]{HambletonKreckTeichner}.
On the other hand,  Kreck and Schafer found pairs of smooth closed~$4$-manifolds with finite $\pi_1$ and isometric equivariant intersection forms that are stably diffeomorphic but not homotopy equivalent~\cite{KreckSchafer}.
When the boundary is nonempty and~$\pi_1=1$,  one can use work of Boyer to produce simply-connected~$4$-manifolds~$M$ with boundary and arbitrarily large (but necessarily finite)~$\mathcal{S}^{\st}_{h,\lambda}(M)$~\cite{BoyerRealization}.
Until now however, there have been no examples of~$4$-manifolds with infinite~$\mathcal{S}^{\st}_{h,\lambda}(M)$.
For every odd prime $q$,  our main result describes a $4$-manifold $M$ with fundamental group $\Z$ and infinite~$\mathcal{S}^{\st}_{h,\lambda_{2q}}(M)$, where the fixed Hermitian form is \[\lambda_{2q} \colon \Z[t^{\pm 1}] \times \Z[t^{\pm 1}] \to \Z[t^{\pm 1}];\; (x,y) \mapsto 2qx\overline{y}. \]

\begin{theorem}
\label{thm:Infinite}
For every odd prime $q$,  there exists an infinite family~$\{M_i\}_{i=1}^{\infty}$ of
  $4$-manifolds with fundamental group $\Z$ that are all stably homeomorphic, and all the manifolds $M_i$ have equivariant intersection pairing
isometric to~$\lambda_{2q}$ and boundary $L(2q, 1) \# (S^1 \times S^2)$, but they are pairwise not homotopy equivalent via any homotopy equivalence that restricts to a homotopy equivalence on the boundary.
In other words,
$$|\mathcal{S}_{h,\lambda_{2q}}^{\st}(M_1)| = \infty.$$
\end{theorem}

For a fixed odd prime $q$, the manifolds in Theorem~\ref{thm:Infinite} all have fundamental group $\Z$,
boundary~\[Y_q := L(2q,1) \# (S^1 \times S^2),\]  equivariant intersection form
isometric to $\lambda_{2q}$, and integral intersection form isometric to
\[\lambda^\Z_{2q} \colon \Z \times \Z \to \Z;\; (x,y) \mapsto 2qxy,\]
but are distinguished by an invariant, first introduced in~\cite{ConwayPiccirilloPowell} and inspired by work of Boyer~\cite{BoyerRealization}, related to the Blanchfield form of~$Y_q$.
While the manifold~$M_1$ is smooth,  we cannot tell whether any of the other $M_i$ admit smooth structures.  Their construction uses surgery methods, in particular a recent realisation result from \cite{ConwayPiccirilloPowell}, which a priori only works in the topological category.


Before giving more details and describing the main steps in the proof of Theorem~\ref{thm:Infinite}, we briefly compare the study of the homotopy stable class in dimension~$4$  with the situation in higher dimensions.

\begin{remark}
\label{rem:HigherDimension}
Kreck and Schafer found pairs of closed smooth~$4k$-manifolds, for~$k \geq 1$, that are stably diffeomorphic and have hyperbolic equivariant intersection forms,  but are pairwise not homotopy equivalent~\cite{KreckSchafer}.
In \cite{CCPS-short}, together with Sixt we gave the first examples of simply-connected, closed,  smooth~$4k$-manifolds, for~$k \geq 2$, with hyperbolic intersection form and arbitrarily large homotopy stable class $\mathcal{S}^{\st}_{h,\lambda}$.
 In \cite{CCPS-long} for $k \geq 2$, we produced smooth closed~$4k$-manifolds with fundamental group~$\Z$, again with hyperbolic intersection form, and such that the homotopy stable class $\mathcal{S}^{\st}_{h,\lambda}$ is infinite.
In those papers we were unable to obtain examples in dimension $4$.
In \cite{CCPS-short}, in lieu of this we defined a $\spin^{c}$ version of the stable class in dimension $4$, and we showed that this spin$^c$ stable
class can be arbitrarily large.
This article shows that a variation on those methods,
with analogous underlying algebra, does produce examples of 4-manifolds with nonempty boundary and fundamental group $\Z$ that have infinite homotopy stable class.
\end{remark}

Next we describe the main steps in the proof of Theorem~\ref{thm:Infinite}.
%
%
Fix an odd prime $q$.
The first observation
is that if~$N_1,N_2$ are~$4$-manifolds with
integral intersection forms
isometric to $\lambda^\Z_{2q}$,
then there can be no orientation-reversing homotopy equivalence between~$N_1$ and~$N_2$.
For this reason,  and for the purpose of proving our main theorem, we restrict to orientation-preserving homotopy equivalences (o.p.  homotopy eq.  for short) and therefore consider
$$\mathcal{S}^{\st}_{h^+,\lambda}(M) := \{N  \mid N \cong_{\st} M, \lambda_N \cong \lambda_M \}/\text{o.p.\ homotopy eq.\ of pairs}.$$
We now restrict to $4$-manifolds $M$ with fundamental group~$\Z$ such that the inclusion $\partial M \subseteq M$ induces a surjection $\varphi \colon \pi_1(\partial M) \twoheadrightarrow \pi_1(M) \xrightarrow{\cong} \Z$  (we say that $M$ has \emph{ribbon boundary}) and for which $H_1(\partial M;\Z[t^{\pm 1}])$ is a $\Z[t^{\pm 1}]$-torsion module.
\purple{Here and throughout the paper we assume that the fundamental groups of our 4-manifolds are equipped with a preferred isomorphism to $\Z$; to indicate this we write~$\pi_1(M) = \Z$.}

\purple{Given two such manifolds $N_1$ and $N_2$, we write $\partial N_1 \cong_B  \partial N_2$ if there exists an orientation-preserving homeomorphism $f \colon \partial N_1 \xrightarrow{\cong} \partial N_2$  that interwines the inclusion induced epimorphisms $\varphi_i \colon \pi_1(\partial N_i) \twoheadrightarrow \pi_1(N_i)$ and,  in the case that $N_1$ and $N_2$ are  spin,  such that the union $N_1 \cup_f -N_2$  is spin.}
The terminology $\cong_B$ is motivated by modified surgery theory~\cite{KreckSurgeryAndDuality}, in which $B$ is the standard notation for the normal 1-type.


Next,  if~$N_1$ and $N_2$ have fundamental group~$\Z$, $\partial N_1 \cong_{\purple{B}} \partial N_2$, the same Kirby-Siebenmann invariant, and~$\lambda_{N_1} \cong \lambda_{N_2}$, then they are stably homeomorphic. Indeed,~$N_1$ and~$N_2$ must have isometric integral intersection forms (in particular with the same type and the same signature) and the same Kirby-Siebenmann invariant, so work of Kreck ensures they are stably homeomorphic~\cite[Theorem 2]{KreckSurgeryAndDuality}\purple{; see Lemma~\ref{lem:Kreck} for details.}

Put differently,  if~$M$ is a~$4$-manifold with infinite cyclic fundamental group, then
\[  \mathcal{S}^{\st}_{h^+,\lambda}(M)= \frac{\{N  \mid \partial N \cong_{\purple{B}} \partial M, \pi_1(N) = \Z,  \lambda_N \cong \lambda_M, \ks(N)=\ks(M) \}}{\text{o.p.~homotopy eq. of pairs}}.\]
This next step is to recast~$\mathcal{S}^{\st}_{h^+,\lambda}(M)$ in terms of the group~$\Aut(\Bl_{\partial M})$ of isometries of the Blanchfield form~$\Bl_{\partial M}$ \purple{(whose definition we recall in Section~\ref{sec:Blanchfield}}).
Firstly,  as we recall in Section~\ref{sec:Reduc},  the group~$\hAut_{\varphi}^+(\partial M)$ of orientation-preserving homotopy equivalences~$h \colon \partial M \simeq \partial M$ that intertwine the inclusion induced map~$\varphi \colon \pi_1(\partial M) \to \pi_1(M)=\Z$ acts on~$\Aut(\Bl_{\partial M})$.
Secondly,  as we also recall in Section~\ref{sec:Reduc},  the group~$\Aut(\lambda_M)$ of isometries of~$\lambda_M$ also acts on~$\Aut(\Bl_{\partial M})$, and the two actions commute with one another.
Quotienting out by these two actions leads to an orbit set~$\Aut(\Bl_{\partial M})/(\Aut(\lambda_M) \times \hAut_{\varphi}^+(\partial M))$.
Note that it need not be group.

In order to account for our $4$-manifolds being spin,  we will in fact need to work with a smaller set of isometries.
Namely,  if $M$ is spin, then $\Bl_{\partial M}$ admits a quadratic enhancement
$$ \mu_{\Bl_{\partial M}} \colon H_1(\partial M;\Z[t^{\pm 1}]) \to \frac{\{b \in \Q(t) \mid b= \ol{b}\}}{\{a + \ol{a} \mid a \in \Z[t^{\pm 1}]\}}$$
and we write $\Aut(\Bl_{\partial M}, \mu_{\Bl_{\partial M}}) \subseteq \Aut(\Bl_{\partial M})$ for those isometries of $\Bl_{\partial M}$ that also preserve~$\mu_{\Bl_{\partial M}}$.
Writing $\hAut_{\varphi}^{+,\purple{\quadr}}(\partial M)$ for those homotopy equivalences whose induced map on the Alexander module preserves $\mu_{\Bl_{\partial M}}$ leads to the orbit set
$$\Aut(\Bl_{\partial M}, \mu_{\Bl_{\partial M}})/(\Aut(\lambda_M) \times \hAut_{\varphi}^{+,\quadr}(\partial M)).$$
One of the main steps in the proof of Theorem~\ref{thm:Infinite} is the following partial description of~$\mathcal{S}^{\st}_{h^+,\lambda}(M)$ for a large class of~$4$-manifolds~$M$ with infinite cyclic fundamental group and \purple{ribbon boundary.}
As we will explain in Proposition~\ref{prop:SurjectionOntoAlgebra}, this result follows fairly promptly from the machinery developed in~\cite{ConwayPiccirilloPowell}.
\purple{In the following proposition, and throughout the paper, \emph{spin} refers to a manifold that admits a spin structure compatible with the orientation.}

\begin{proposition}
\label{prop:SurjectionOntoAlgebraIntro}
If~$M$ is a \purple{spin}~$4$-manifold with ribbon boundary,~$\pi_1(M)= \Z$, and nondegenerate equivariant intersection form~$\lambda_M$, then there is a surjection
\color{black}
$$ b\colon  \mathcal{S}_{h^+ ,\lambda}^{\st}(M) \twoheadrightarrow  \Aut(\Bl_{\partial M},\purple{\mu_{\Bl_{\partial M}}})/(\Aut(\lambda_M)\times \hAut_{\varphi}^{+,\purple{\quadr}}(\partial M)).$$
\color{black}
The surjection is described explicitly in Construction~\ref{cons:CPP22}.
\end{proposition}

Fix an odd prime $q$ and let~$X_{2q}(U)$ denote the~$2q$-trace on the unknot~$U$, i.e.\ the smooth~$4$-manifold obtained from~$D^4$ by attaching a~$2q$-framed~$2$-handle along the unknot.
The final part of the proof of Theorem~\ref{thm:Infinite}, which is carried out in Proposition~\ref{thm:InfiniteAlgebra}, consists of proving that for~$M=X_{2q}(U) \natural (S^1 \times D^3)$, the set
$\purple{\Aut (\Bl_{\partial M}, \mu_{\Bl_{\partial M}})/( \Aut(\lambda_M) \times \hAut_{\varphi}^{+,\quadr}(\partial M))}$
is countably infinite.
Modulo this statement, we can now conclude the proof of Theorem~\ref{thm:Infinite}, which states that~$\mathcal{S}_{h,\lambda_{2q}}^{\st}(M)$ is infinite.

\begin{proof}[Proof of Theorem~\ref{thm:Infinite}]
Fix an odd prime $q$ and
consider~$M:=X_{2q}(U) \natural (S^1 \times D^3)$.
\purple{This $4$-manifold is spin, has ribbon boundary,  admits an identification $\pi_1(M) = \Z$, and has nondegenerate equivariant intersection $\lambda_M\cong (2q)$.}
Since for any two 4-manifolds~$N_1$ and $N_2$ with 
integral intersection forms isometric to $\lambda^\Z_{2q}$,
there is no orientation reversing homotopy equivalence between them,
$\mathcal{S}_{h,\lambda_{2q}}^{\st}(M_{})=\mathcal{S}_{h^+,\lambda_{2q}}^{\st}(M_{})$.
We therefore prove that~$\mathcal{S}_{h^+,\lambda_{2q}}^{\st}(M_{})$ is infinite.
To prove this we apply Proposition~\ref{prop:SurjectionOntoAlgebraIntro}, which implies
that~$\mathcal{S}_{h^+,\lambda_{2q}}^{\st}(M_{})$ surjects
onto \purple{the orbit set}~$\purple{\Aut (\Bl_{\partial M}, \mu_{\Bl_{\partial M}})/( \Aut(\lambda_M) \times \hAut_{\varphi}^{+,\quadr}(\partial M))}$
and this latter set is countably infinite by Proposition~\ref{thm:InfiniteAlgebra}.
\end{proof}

\begin{remark}
\label{rem:Finite}
The existence of $M$ with infinite $\Aut(\Bl_{\partial M},\mu_{\Bl_{\partial  M}})/ (\Aut(\lambda_M) \times \hAut_{\varphi}^{+,\purple{\quadr}}(\partial M))$ is what makes it possible for us to obtain an example where the homotopy stable class $\mathcal{S}_{h,\lambda}^{\st}$ is infinite.
While an analogue of Proposition~\ref{prop:SurjectionOntoAlgebraIntro} can be proved in the simply-connected case using results of Boyer~\cite{BoyerRealization},  the corresponding algebra always remains finite for trivial fundamental group.

All of the infinite sets we discuss are necessarily countable.  Primarily, this has to be the case because there are only countably many compact manifolds~\cite{CheegerKister}.
 On the algebraic side it is also evident that the orbit set onto which the homotopy stable class surjects in Proposition~\ref{prop:SurjectionOntoAlgebraIntro} is countable, essentially because all the homology groups involved are finitely generated over $\Z[t^{\pm 1}]$.
\end{remark}

Next we  discuss a variation on Proposition~\ref{prop:SurjectionOntoAlgebraIntro} that may be of independent interest.
%
%
The surjection in Proposition~\ref{prop:SurjectionOntoAlgebraIntro} can be improved to a bijection if we require the homotopy equivalences~$N_1 \simeq N_2$ to restrict to homeomorphisms on the boundary; i.e.\ if we consider
\[ \mathcal{S}^{\st,\partial}_{h^+,\lambda}(M) := \frac{ \{N  \mid N \cong_{\st} M,\,  \lambda_N \cong \lambda_M \}}{\text{o.p.\ homotopy eq.\ that restricts to a homeo.\  on the boundary}}\]
and change the target accordingly, i.e.\ consider~$\Aut(\Bl_{\partial M},\mu_{\Bl_{\partial M}})/(\Aut(\lambda_M) \times \Homeo_\varphi^{+,\purple{\quadr}}(\partial M))$ instead of~$\Aut(\Bl_{\partial M},\mu_{\Bl_{\partial M}})/(\Aut(\lambda_M) \times \hAut_{\varphi}^{+,\quadr}(\partial M))$.
In fact, the same result is obtained with
\[  \mathcal{S}^{\st}_{+,\lambda}(M) := \frac{ \{N  \mid  N \cong_{\st} M,\,  \lambda_N \cong \lambda_M\}}{\text{o.p. homeomorphism}}.\]

%
%
%

\begin{proposition}
\label{prop:Bijection}
 If~$M$ is a \purple{spin}~$4$-manifold with~$\pi_1(M) = \Z$,  ribbon boundary and nondegenerate equivariant intersection form~$\lambda_M$,  then there are bijections:
\begin{align*}
\mathcal{S}_{h^+,\lambda}^{\st,\partial}(M) &\xrightarrow{\approx}  \Aut(\Bl_{\partial M},\purple{\mu_{\Bl_{\partial M}}})/(\Aut(\lambda_M)\times \Homeo_\varphi^{+,\purple{\quadr}}(\partial M)),\text{ and} \\
\mathcal{S}^{\st}_{+,\lambda}(M) &\xrightarrow{\approx}  \Aut(\Bl_{\partial M},\purple{\mu_{\Bl_{\partial M}}})/(\Aut(\lambda_M)\times \Homeo_\varphi^{+,\purple{\quadr}}(\partial M)).
\end{align*}
The bijections are induced by the map~$b$ that will be introduced in Construction~\ref{cons:CPP22}.
For $M:=X_{2q}(U) \natural (S^1 \times D^3)$, with $q$ an odd prime, the sets above are countably infinite.
\end{proposition}

\begin{proof}
The surjectivity follows from the same argument that we will use in Proposition~\ref{prop:SurjectionOntoAlgebra}. We prove injectivity.
 If~$b(N_1) = b(N_2)$, then~\cite[Theorem 1.1]{ConwayPiccirilloPowell} shows that the manifolds~$N_1$ and~$N_2$ are orientation-preserving homeomorphic.  Since the quotient with $\Homeo_\varphi^{+,\purple{\quadr}}(\partial M)$ replaced by $\hAut_\varphi^{+,\purple{\quadr}}(\partial M)$ is infinite, and since $\Homeo_\varphi^{+,\purple{\quadr}}(\partial M) \subseteq \hAut_\varphi^{+,\purple{\quadr}}(\partial M)$, it follows that the sets in the statement are infinite.
\end{proof}

We conclude the introduction by characterising~$M := X_{2q}(U) \natural (S^1 \times D^3)$ within $\mathcal{S}^{\st}_{+,\lambda}(M)$ in terms of the knottedness of the sphere $S^2_l := \lbrace \mathrm{pt}\rbrace \times S^2 \subseteq ((S^1 \times S^2) \setminus \mathrm{Int}(D^3)) \subseteq
\partial M$ and the connect sum sphere $S^2_c \subseteq M$.

\begin{theorem}
For $M =X_{2q}(U) \natural (S^1 \times D^3)$ and $N \in \mathcal{S}^{\st}_{+,\lambda}(M)$,  the following are equivalent:
\begin{enumerate}
\item $N$ is homeomorphic to $M$;
\item $S^2_l \subseteq \partial N$ bounds a locally flat $D^3 \subseteq N$;
\item $S^2_c$ bounds a locally flat $D^3 \subseteq N$.
\end{enumerate}
\end{theorem}

\begin{proof}
The implications $1) \Rightarrow 2)$  and $1) \Rightarrow 3)$ are immediate.

We prove the implication $2) \Rightarrow 1)$.
Cut $N$ along the $D^3$ with boundary $S^2_l$ to obtain a simply connected $4$-manifold with boundary $L(2q,1)$ and
$H_2 = \Z$.
Work of Boyer implies that such a manifold is homeomorphic to $X_{2q}(U)$~\cite[Theorem 0.1]{BoyerUniqueness}.
Glue back the~$D^3 \times [0,1]$ that we removed to recover $N$ as $M$.

Finally, we prove the implication $3) \Rightarrow 1)$.
Cut $N$ open along the separating $D^3$, resulting in a disjoint union of two $4$-manifolds.
The first  is simply connected with $H_2 = \Z$ and boundary~$L(2q,1)$ and is therefore homeomorphic to $X_{2q}(U)$~\cite[Theorem 0.1]{BoyerUniqueness}.
The second has $\pi_1=\Z$, no~$H_2$ and boundary $S^1 \times S^2$; it is thus homeomorphic to~$S^1 \times D^3$~\cite[Section 11.6]{FreedmanQuinn}.
Glue back the~$D^3 \times [0,1]$ that we removed to recover~$N$ as $M$.
\end{proof}

\subsection*{Organisation}

\purple{In Sections~\ref{sec:LinkingForms} and \ref{sec:Blanchfield}, we review some facts about linking forms and in particular the Blanchfield form. In Section~\ref{sec:StableHomeo} we give a criterion that implies stable homeomorphism of $4$-manifolds with fundamental group $\Z$ and nonempty boundary.}
In Section~\ref{sec:Reduc}, we prove Proposition~\ref{prop:SurjectionOntoAlgebraIntro}. In Section~\ref{sec:InfinitebAut} we show that for~$M=X_{2q}(U) \natural (S^1 \times D^3)$,  with $q$ an odd prime, the set~$\Aut(\Bl_{\partial M},\mu_{\Bl_{\partial M}})/(\Aut(\lambda_M) \times \hAut_{\varphi}^{+,\purple{\quadr}}(\partial M))$ is infinite.

\subsection*{Conventions}
\label{sub:Conventions}

We work in the topological category  unless otherwise stated.
All manifolds are assumed to be compact, connected, based, and oriented. If a manifold has a nonempty boundary, then the basepoint is assumed to be in the boundary.  \purple{For a 4-manifold $M$ with fundamental group $\Z$, we fix an identification an write $\pi_1(M) =\Z$.  We say that $M$ is spin if $M$ admits a spin structure compatible with the orientation.}
\purple{We write~$p \mapsto \overline{p}$ for the involution on~$\Z[t^{\pm 1}]$ induced by~$t \mapsto t^{-1}$.
Given a~$\Z[t^{\pm 1}]$-module~$H$, we write~$\overline{H}$ for the~$\Z[t^{\pm 1}]$-module whose underlying abelian group is~$H$ but with module structure given by~$p \cdot h=\overline{p}h$ for~$h \in H$ and~$p \in \Z[t^{\pm 1}]$.  We write $H^*:=\overline{\Hom_{\Z[t^{\pm 1}]}(H,\Z[t^{\pm 1}])}$.}


\subsection*{Acknowledgements}

We thank the referee for helpful comments which helped improve the paper.
MP was partially supported by the EPSRC New Investigator grant EP/T028335/2 and EPSRC New Horizons grant EP/V04821X/2.

\section{Linking forms and unions}
\label{sec:LinkingForms}

Since a large part of this paper is concerned with the Blanchfield form and isometries thereof, we start by recalling terminology related to the underlying algebra.
In Section~\ref{sub:Algebra} we recall symmetric and quadratic linking forms. In Section~\ref{sub:Boundary} we recall how a Hermitian form has a boundary which is a symmetric linking form, and the boundary of an even form has the additional structure of a quadratic refinement.   In Section~\ref{sub:Unions} we recall how isometries of these linking forms can be used to glue two linking forms together, and we show that the union of two even forms along an isometry of their boundary quadratic linking forms is again an even form.

\subsection{Symmetric and quadratic linking forms}
\label{sub:Algebra}

Everything in this subsection is the special case for $\Z[t^{\pm 1}]$ of a general theory for arbitrary rings with involution  developed by Ranicki~\cite[\S3.4]{RanickiExact}.


\begin{definition}
\label{def:SymmetricLinkinForm}
  A \emph{symmetric linking form} over $\Z[t^{\pm 1}]$ is a pair $(T,\ell)$, where $T$ is a torsion $\Z[t^{\pm 1}]$-module, and $\ell \colon T \times T \to \Q(t)/\Z[t^{\pm 1}]$ is a Hermitian, sesquilinear, nonsingular pairing.
\end{definition}

We write $S:= \Z[t^{\pm 1}] \sm \{0\}$, and set
\begin{align*}
  Q^1(\Q(t)/\Z[t^{\pm 1}]) &:= \frac{\{b \in \Q(t) \mid b-\ol{b} \in \Z[t^{\pm 1}]\}}{\Z[t^{\pm 1}]}, \\
  Q_1(\Z[t^{\pm 1}],S) &:= \frac{\{b \in \Q(t) \mid b= \ol{b}\}}{\{a + \ol{a} \mid a \in \Z[t^{\pm 1}]\}},\\
  Q^1(\Z[t^{\pm 1}],S) &:= \frac{\{b \in \Q(t) \mid b-\ol{b} = a - \ol{a} \text{ for some } a \in \Z[t^{\pm 1}]\}}{\Z[t^{\pm 1}]} \subseteq Q^1(\Q(t)/\Z[t^{\pm 1}]).
\end{align*}

For a symmetric linking form $(T,\ell)$, we have that $\ell(x,x) \in Q^1(\Q(t)/\Z[t^{\pm 1}])$ for all $x \in T$. The symmetric linking form is called \emph{even} if $\ell(x,x) \in  Q^1(\Z[t^{\pm 1}],S)$ for all $x \in T$.
We define a map
\begin{align*}
  q \colon  Q_1(\Z[t^{\pm 1}],S)  &\to  Q^1(\Z[t^{\pm 1}],S)\\
  [b] &\mapsto [b].
\end{align*}

\begin{definition}\label{defn:quadratic-refinement}
  A \emph{quadratic refinement} of an even symmetric linking form $(T,\ell)$ is a function $\mu \colon T \to Q_1(\Z[t^{\pm 1}],S)$ satisfying
  \begin{enumerate}[(i)]
    \item\label{item:quad-refinement-i} $\mu(rx) = r\mu (x) \ol{r} \in  Q_1(\Z[t^{\pm 1}],S)$ for all $x \in T$ and for all $r \in \Z[t^{\pm 1}]$;
    \item\label{item:quad-refinement-ii} $\mu(x+y) = \mu(x) + \mu(y) + \ell(x,y) + \ol{\ell(x,y)} \in  Q_1(\Z[t^{\pm 1}],S)$ for all $x,y \in T$;
    \item\label{item:quad-refinement-iii} $q(\mu(x)) = \ell(x,x) \in  Q^1(\Z[t^{\pm 1}],S)$ for all $x \in T$.
  \end{enumerate}
  A triple $(T,\ell,\mu)$ consisting of a symmetric linking form together with a quadratic refinement is called a \emph{quadratic linking form} over $\Z[t^{\pm 1}]$.
\end{definition}

\purple{For aficionados of~\cite{RanickiExact},  we emphasise that we are using the non-split version of quadratic linking forms.}

We will also need to consider isometries and automorphisms of symmetric and quadratic linking forms.

\begin{definition}
  Let $(T,\ell)$ and $(T',\ell')$ be symmetric linking forms over $\Z[t^{\pm 1}]$ and let $\mu \colon T \to Q_1(\Z[t^{\pm 1}],S)$ and $\mu' \colon T' \to Q_1(\Z[t^{\pm 1}],S)$ be respective quadratic refinements.
  \begin{enumerate}
 \item\label{item:1-defn-auts} An isomorphism
$f \colon T \to T'$ is
 an \emph{isometry of symmetric linking forms}
 if $$\ell'(f(x),f(y)) = \ell(x,y)$$ for every $x,y \in T$.
    \item\label{item:2-defn-auts} The isometry of symmetric linking forms~$f$ is moreover an \emph{isometry of quadratic linking forms}, $f \colon (T,\ell,\mu) \cong (T',\ell',\mu')$ if $\mu'(f(x)) = \mu(x)$ for every $x \in T$.
    \item\label{item:3-defn-auts} If $(T,\ell) = (T',\ell')$, then $f$ as in \eqref{item:1-defn-auts} is an \emph{automorphism of symmetric linking forms}. We write $\Aut(T,\ell)$ for the group of automorphisms.
    \item\label{item:4-defn-auts} If $(T,\ell,\mu) = (T',\ell',\mu')$, then $f$  as in \eqref{item:2-defn-auts} is an \emph{automorphism of quadratic linking forms}. We write $\Aut(T,\ell,\mu)$ for the group of automorphisms.
  \end{enumerate}
\end{definition}

\begin{remark}\label{remark:inclusion-of-automorphisms-quad-symm}
  Given a quadratic linking form $(T,\ell,\mu)$ over $\Z[t^{\pm 1}]$ with underlying symmetric linking form $(T,\ell)$,  we note that~$\Aut(T,\ell,\mu) \subseteq \Aut(T,\ell)$.
  We give an example showing that this can be a proper inclusion: multiplication by $3$ induces an isomorphism $\Z[t^{\pm 1}]/8 \to \Z[t^{\pm 1}]/8$ that preserves the linking form $\ell(x,y)=\frac{1}{8}x\overline{y}$ but does not preserve the quadratic refinement $\mu(x)=\frac{1}{8}x$.
Indeed $\mu(3)=\frac{9}{8} \neq \frac{1}{8}=\mu(1) \in Q_1(\Z[t^{\pm 1}],S)$ because~$1$ cannot be written as $a+\overline{a}$ with $a \in \Z[t^{\pm 1}]$.
\end{remark}


\subsection{Boundaries of quadratic forms}
\label{sub:Boundary}

We recall some terminology about Hermitian forms.
\purple{A \emph{Hermitian form} refers to a pair $(H,\lambda)$ where $H$ is a free $\Z[t^{\pm 1}]$-module and $\lambda \colon H \times H \to \Z[t^{\pm 1}]$ is a sesquilinear Hermitian pairing.}
Given a Hermitian form~$(H,\lambda)$ over $\Z[t^{\pm 1}]$, we use $\widehat{\lambda} \colon H \to H^*=:\overline{\Hom_{\Z[t^{\pm 1}]}(H,\Z[t^{\pm 1}])}$ to denote the linear map defined by~$\widehat{\lambda}(y)(x)=\lambda(x,y)$.
We often refer to $\widehat{\lambda}$ as the \emph{adjoint} of $\lambda$.
We say that $\lambda$ is \emph{nondegenerate} if $\widehat{\lambda}$ is injective and \emph{nonsingular} if $\widehat{\lambda}$ is an isomorphism.
We also recall that a Hermitian form $(H,\lambda)$ is called \emph{even} if for all $x \in H$, there exists $a \in  \Z[t^{\pm 1}]$ such that~$\lambda(x,x) = a + \ol{a}$.

We describe how a nondegenerate even Hermitian form over~$\Z[t^{\pm 1}]$ determines a quadratic linking form,  following \cite[p.~243]{RanickiExact}.

\begin{definition}
\label{def:BoundaryLinkingForm}
The \emph{boundary symmetric linking form} of a nondegenerate Hermitian form~$(H,\lambda)$ over~$\Z[t^{\pm 1}]$ is the symmetric linking form~$(\coker(\widehat{\lambda}),\partial \lambda)$, where~$\partial \lambda$ is defined as
\begin{align*}
 \partial \lambda \colon \coker(\widehat{\lambda}) \times \coker(\widehat{\lambda}) &\to \Q(t)/\Z[t^{\pm 1}]  \\
 ([x],[y]) &\mapsto  \smfrac{1}{s}(y(z)),
\end{align*}
where, as~$\coker(\widehat{\lambda})$ is~$\Z[t^{\pm 1}]$-torsion, there exists~$s \in \Z[t^{\pm 1}]$ and~$z \in H$ such that~$sx=\widehat{\lambda}(z)$.

If $(H,\lambda)$ is additionally assumed to be even, then its \emph{boundary quadratic linking form} is the quadratic linking form~$(\coker(\widehat{\lambda}),\partial \lambda, \mu_{\partial})$, where the quadratic refinement of $\partial \lambda$ is
\begin{align*}
\mu_{\partial} \colon \coker(\widehat{\lambda}) &\to Q_1(\Z[t^{\pm 1}],S)  \\
 [y] &\mapsto  \smfrac{1}{s}(y(z)),
\end{align*}
with $s \in \Z[t^{\pm 1}]$ and~$z \in H$ such that~$sy=\widehat{\lambda}(z)$.
\end{definition}

We assert that~$\partial \lambda$ is independent of the choices involved, and is nonsingular, sesquilinear,  and Hermitian. These can all be verified directly.
To enable us to give a reference to existing literature, note that the boundary symmetric linking form is the linking form on~$H^1(C^*)$ associated with the 1-dimensional $\Q(t)$-acyclic symmetric Poincar\'{e} complex~$(C_*,\varphi)$ over~$\Z[t^{\pm 1}]$ given by
\[\xymatrix{
  C^0 = H \ar[r]^-{\lambda} \ar[d]^{\varphi_0=\id} & C^1 = H^* \ar[d]^{\varphi_0=\id} \\
    C_1 = H \ar[r]^-{\lambda^* = \lambda}  & C_0 = H^*.}\]
In \cite[Propositions~3.3,~3.4,~and~3.8]{PowellBlanchfield} it was shown that such a linking form is well-defined, nonsingular, sesquilinear,  and Hermitian.

The next proposition is implicit in \cite[p.~243]{RanickiExact}. As far as we know such a proof has no appeared in the literature, so for the convenience of the reader we provide the details of the proof.

\begin{proposition}
The function~$\mu_{\partial}$ is a well-defined function $\coker(\widehat{\lambda}) \to Q_1(\Z[t^{\pm 1}],S)$, and is a quadratic refinement of the boundary symmetric linking form $\partial \lambda$ of $(H,\lambda)$, i.e.\ $\mu_{\partial}$ satisfies the requirements of Definition~\ref{defn:quadratic-refinement}.
  \end{proposition}
\begin{proof}
  First we show it is well-defined.  Let $y$, $z$, and $s$ be as in Definition~\ref{def:BoundaryLinkingForm}. Express $\widehat{\lambda}$ as a Hermitian matrix $A = \ol{A}^T$ over $\Z[t^{\pm 1}]$ with respect to some basis of $H$. Then $A$ is invertible over $\Q(t)$ because $\lambda$ is nondegenerate, so $\det(A)\neq 0$.  We have $\mu_{\partial}(y) = y^TA^{-1}\ol{y}$.
  Therefore
  \[\ol{\mu_{\partial}(y)} = \ol{\mu_{\partial}(y)}^T  = (\ol{y^TA^{-1}\ol{y}})^T = y^T\ol{A^{-1}}^T\ol{y} = y^T(\ol{A}^T)^{-1}\ol{y} = y^TA^{-1}\ol{y} = \mu_{\partial}(y).\]
  Hence $\mu_{\partial}(y) \in Q_1(\Z[t^{\pm 1}],S)$.

  Next we show that the choices of $z$ and $s$ do not change $\mu_\partial(y)$. Let $y \in H$ and let $y \in  \coker(\widehat{\lambda})$, $s\in \Z[t^{\pm 1}]$, and $z \in H$ be as in Definition~\ref{def:BoundaryLinkingForm}, with $\widehat{\lambda}(z) =sy$. Let  $s' \in \Z[t^{\pm 1}]$ and $z' \in H$ be another pair of choices, such that $\widehat{\lambda}(z') = s'y$.  Since  $\mu_{\partial}(y) \in Q_1(\Z[t^{\pm 1}],S)$, we have $\smfrac{1}{s'} y(z') = \ol{\smfrac{1}{s'} y(z')}$. The difference between the two  computations of  $\mu_{\partial}(y)$ yields
  \begin{align*}
     \smfrac{1}{s}y(z) - \smfrac{1}{s'}y(z') =& \smfrac{1}{s}y(z) - \ol{\smfrac{1}{s'}y(z')}
     = \smfrac{1}{s}y(z)\smfrac{\ol{s'}}{\ol{s'}} - \ol{\smfrac{1}{s'}y(z')\smfrac{\ol{s}}{\ol{s}}}
     = \smfrac{1}{s}(s'y)(z)\smfrac{1}{\ol{s'}} - \ol{\smfrac{1}{s'}(sy)(z')\smfrac{1}{\ol{s}}}\\
     =& \smfrac{1}{s}\widehat{\lambda}(z')(z)\smfrac{1}{\ol{s'}} - \ol{\smfrac{1}{s'}\widehat{\lambda}(z)(z')\smfrac{1}{\ol{s}}}
     = \smfrac{1}{s}\lambda(z,z')\smfrac{1}{\ol{s'}} - \ol{\smfrac{1}{s'}\ol{\lambda(z,z')}\smfrac{1}{\ol{s}}} = 0 \in Q_1(\Z[t^{\pm 1}],S).
  \end{align*}

Next we show that $\mu_{\partial}$ does not depend on the representative $y$ for the class in $\coker(\widehat{\lambda})$. Replace $y \in \coker(\widehat{\lambda})$ by another representative $y + \widehat{\lambda}(u)$, for some $u \in H$. Then $\widehat{\lambda}(z+su) = s(y+ \widehat{\lambda}(u))$.
Therefore
\[\mu_{\partial}(y + \widehat{\lambda}(u)) = \smfrac{1}{s}(y+ \widehat{\lambda}(u))(z+su) = \smfrac{1}{s} y(z) + \smfrac{1}{s} \widehat{\lambda}(u)(z) + y(u) +\widehat{\lambda}(u)(u).\]
We have \[\widehat{\lambda}(u)(z) = \lambda(z,u) = \ol{\lambda(u,z)} = \ol{\widehat{\lambda}(z)(u)} = \ol{(sy)(u)} = \ol{y(u)\ol{s}} = s \ol{y(u)}.\]
Substituting, we obtain that \[\mu_{\partial}(y + \widehat{\lambda}(u)) = \smfrac{1}{s} y(z) + \ol{y(u)} + y(u) +\lambda(u,u).\]
The last term is symmetric over $\Z[t^{\pm 1}]$, which implies it is of the form $a + \ol{a}$. Hence up to terms of the form $a + \ol{a}$, we have $\mu_{\partial}(y + \widehat{\lambda}(u)) = \smfrac{1}{s} y(z) = \mu_{\partial}(y)$, as desired.

Now we know that $\mu_{\partial}$ is well-defined, we prove that it satisfies the conditions in Definition~\ref{defn:quadratic-refinement} for it to be a quadratic refinement of the boundary symmetric linking form $\partial \lambda$.
For \eqref{item:quad-refinement-i}, let $y \in \coker(\widehat{\lambda})$, let $r \in \Z[t^{\pm 1}]$, and let $z \in H$ be such that $\widehat{\lambda}(z)= sy$. Then $\widehat{\lambda}(rz) = rsy= sry$. Thus \[\mu_{\partial}(ry) =  \smfrac{1}{s}((ry)(rz)) = \smfrac{r}{s}(y(z))\ol{r} = r\mu_{\partial}(y)\ol{r},\] as desired.
Next, aiming for~\eqref{item:quad-refinement-ii}, we compute $\mu_{\partial}(x+y)$, for $x,y \in \coker(\widehat{\lambda})$. Let $r,s \in \Z[t^{\pm 1}]$ and $w,z \in H$ be such that $\widehat{\lambda}(w)= rx$ and $\widehat{\lambda}(z)= sy$. Then $\widehat{\lambda}(sw+ry)= srx + rsy = rs(x+y)$. Hence \begin{align*}
\mu_\partial(x+y) &= \smfrac{1}{rs}(x+y)(sw+rz) = \smfrac{1}{r}x(w) + \smfrac{1}{s}(y)(z) + \smfrac{1}{r}y(w) + \smfrac{1}{s}x(z)  \\ &= \mu_{\partial}(x) + \mu_{\partial}(y)  + \partial \lambda(x,y) +  \partial \lambda(y,x).
\end{align*}
Since $\partial \lambda$ is Hermitian, this proves \eqref{item:quad-refinement-ii}.
Condition \eqref{item:quad-refinement-iii} is immediate from the formulae.
\end{proof}

%
%
%
%

\begin{remark}
\label{rem:ActionIsometry}
We note for later use that an isomorphism~$F \colon H_0 \to H_1$ induces an isomorphism~$F^{-*}:= (F^*)^{-1} \colon H_0^* \to H_1^*$ and that if additionally, the isomorphism~$F$ is an isometry, then~$F^{-*}$ descends to an isomorphism
\[\partial F:=F^{-*} \colon \coker(\widehat{\lambda}_0)\to \coker(\widehat{\lambda}_1)\]
which determines an isometry of quadratic linking forms.
It follows that $\Aut(\lambda)$ acts both on $\Aut(\coker(\widehat{\lambda}),\partial \lambda)$ and on the subset~$\Aut(\coker(\widehat{\lambda}),\partial \lambda,\partial \mu_{\partial})$ by $F \cdot h=h \circ \partial F^{-1}$.
\end{remark}

\subsection{Algebraic unions}
\label{sub:Unions}

We recall the definition of the union of two Hermitian forms along an isometry of their boundary linking forms.
\purple{The definition appears for the ring $\Z$ in~\cite[Lemma 3.6]{CrowleyThesis} and was generalised to the
ring $\Z[t^{\pm 1}]$ in~\cite[Construction 2.7]{ConwayPowell}.}
The goal of this section is to prove that if the isometry preserves the quadratic refinements,
then the union is an even form.

\begin{construction}
\label{construction:Gluing}
Let~$(H_0,\lambda_0)$ and~$(H_1,\lambda_1)$ be nondegenerate Hermitian forms over~$\Z[t^{\pm 1}]$, and
let~$h \colon (\coker(\widehat{\lambda}_0),\partial \lambda_0) \to (\coker(\widehat{\lambda}_1),\partial \lambda_1)$ be an isometry of their boundary symmetric linking forms.
Consider the pair~$(H_0 \cup_h H_1,\lambda_0 \cup_h -\lambda_1)$ with
\begin{align*}
&H_0 \cup_h H_1:=\ker \big(h \pi_0-\pi_1 \colon H_0^* \oplus H_1^* \to \coker(\widehat{\lambda}_1) \big)  \\
 &\lambda_0 \cup_h -\lambda_1 \left( \bsm x_0 \\ x_1 \esm, \bsm y_0 \\ y_1  \esm \right) =\smfrac{1}{s_0}y_0(z_0)-\smfrac{1}{s_1}y_1(z_1) \in \Q(t),
 \end{align*}
where, since~$\coker(\widehat{\lambda}_i)$ is torsion, there exists~$s_i \in \Z[t^{\pm 1}]$ and~$z_i \in H_i$ such that~$s_ix_i=\widehat{\lambda}_i(z_i)$.
Since the Hermitian forms~$\lambda_0$ and~$\lambda_1$ are nondegenerate, it is not difficult to prove that the pairing~$\lambda_0 \cup_h -\lambda_1$ does not depend on the choice of~$s_0,s_1,z_0,z_1$.
One verifies that $\lambda_0 \cup_h -\lambda_1$ is a sesquilinear, Hermitian form and takes values in $\Z[t^{\pm 1}]$; see~\cite[Proposition 2.8]{ConwayPowell}.
This pairing will be referred to as the \emph{algebraic union} of $\lambda_0$ and $\lambda_1$.
\end{construction}

\begin{lemma}
\label{lem:PropertiesUnion}
Let~$(H_0,\lambda_0),(H_1,\lambda_1)$ and $(H,\lambda)$ be nondegenerate Hermitian forms over~$\Z[t^{\pm 1}]$, and
let~$h \colon (\coker(\widehat{\lambda}_0),\partial \lambda_0) \to (\coker(\widehat{\lambda}_1),\partial \lambda_1)$ be an isometry of the boundary linking forms.
If $F \colon \lambda_0 \cong \lambda_1$ is an isometry,  then there is an isometry
$$\lambda_0 \cup_{h} -\lambda_1 \cong \lambda_1 \cup_{h \circ \partial F^{-1}} -\lambda_1.$$
\end{lemma}
\begin{proof}
 See \cite[Proposition 2.8]{ConwayPowell}.   
\end{proof}

\begin{lemma}
\label{lem:UnionEven}
Let $(H_0,\lambda_0)$ and~$(H_1,\lambda_1)$ be two nondegenerate even Hermitian forms over~$\Z[t^{\pm 1}]$.
  Suppose that we have an isometry $F \colon (H_0,\lambda_0) \cong (H_1,\lambda_1)$ and that
  \[h \colon
  (\coker(\widehat{\lambda}_0),\partial \lambda_0,(\mu_{\partial})_0) \to (\coker(\widehat{\lambda}_1),\partial \lambda_1,(\mu_{\partial})_1)\]
   is an isometry of quadratic linking forms. Then the algebraic union $\lambda_0 \cup_h -\lambda_1$ is even.
\end{lemma}

\begin{proof}
Using Lemma~\ref{lem:PropertiesUnion} we can assume without loss of generality that $H_0=H_1$ and $\lambda_0=\lambda_1$
and $(\mu_{\partial})_0 = (\mu_{\partial})_1$.
Write them both as $(H,\lambda,\mu_{\partial})$.
This means composing $h$ with $\partial F$, but since $\partial F$ is an isometry of quadratic linking forms too, this does not affect the argument.  We abuse notation and without loss of generality use $h$ to denote the new isometry of boundary quadratic linking forms.

  It therefore suffices to check that for every $x_0,x_1 \in H^*$ such that $h \circ \pi_0(x_0) = \pi_1(x_1)$, the self-intersection  $\lambda \cup_h -\lambda \big((x_0,x_1),(x_0,x_1) \big)$ is of the form $a+\overline{a}$ for some $a \in \Z[t^{\pm 1}]$.

Pick $z_0,z_1 \in H$ and $s_0,s_1 \in \Z[t^{\pm 1}]$ such that $s_0x_0=\widehat{\lambda}(z_0)$ and $s_1x_1=\widehat{\lambda}(z_1)$.
This implies both that $\mu_\partial(\pi_i(x_i))=\frac{1}{s_i}(x_i(z_i)) \in Q_1(\Z[t^{\pm 1}],S)$ and that
$$\lambda \cup_h -\lambda \big((x_0,x_1),(x_0,x_1) \big)=\smfrac{1}{s_0}x_0(z_0) - \smfrac{1}{s_1}x_1(z_1) \in \Z[t^{\pm 1}].$$
Passing to $Q_1(\Z[t^{\pm 1}],S)$, by the definition of $\mu_\partial$ this equals \[\mu_{\partial}(\pi_0(x_0)) - \mu_{\partial}(\pi_1(x_1)) =
\mu_{\partial}(\pi_0(x_0)) - \mu_{\partial}(h \circ \pi_0(x_0)) = 0.\]
The first equality used  $h \circ \pi_0(x_0) = \pi_1(x_1)$. The second equality used that $h$ is an isometry of quadratic linking forms.

Now we just have to note that the indeterminacy in $Q_1(\Z[t^{\pm 1}],S)$ consists entirely of even elements $\{a + \ol{a} \mid a \in \Z[t^{\pm 1}]\}$, and therefore $\lambda \cup_h -\lambda \big((x_0,x_1),(x_0,x_1) \big) \subseteq \{a + \ol{a} \mid a \in \Z[t^{\pm 1}]\}$.  It follows that $\lambda \cup_h -\lambda$ is even, as desired.
\end{proof}

\section{The Blanchfield form}
\label{sec:Blanchfield}


In Section~\ref{sub:bl-defn} we review the definition of the Blanchfield form and how it is related to the equivariant intersection form of a $4$-manifold with fundamental group~$\Z$.
Then in Section~\ref{sub:HomotopyEquivalencesIsometry} we review isometries of Blanchfield forms. In Section~\ref{sub:unions-of-spin} we prove promised the spin gluing result, which demonstrates how isometries of the Blanchfield form together with a quadratic refinement can be used to ensure that the union of two spin $4$-manifolds is again spin.

\subsection{The Blanchfield form}\label{sub:bl-defn}
We recall the definition of the Blanchfield form $\Bl_Y$ of a closed $3$-manifold $Y$ \purple{equipped with an epimorphism $\varphi \colon \pi_1(Y) \twoheadrightarrow \Z$} and how, if $M$ is a $4$-manifold with ribbon boundary,  then $\Bl_{\partial M}$ is related to the equivariant intersection form~$\lambda_M$ of $M$.

\begin{construction}
Let $Y$ be a closed $3$-manifold and let $\varphi \colon \pi_1(Y) \twoheadrightarrow \Z$ be an epimorphism.
Assume that the Alexander module $H_1(Y;\Z[t^{\pm 1}])$ is torsion so that, in particular, the Bockstein homomorphism $\operatorname{BS} \colon H^1(Y;\Q(t)/\Z[t^{\pm 1}]) \to H^2(Y;\Z[t^{\pm 1}])$ is an isomorphism.
Consider the composition of Poincar\'e duality, the inverse Bockstein homomorphism and the evaluation homomorphism:
\begin{align*}
\Phi \colon H_1(Y;\Z[t^{\pm 1}]) &\xrightarrow{\operatorname{PD, \cong}} H^2(Y;\Z[t^{\pm 1}]) \xrightarrow{\operatorname{BS}^{-1},\cong} H^1(Y;\Q(t)/\Z[t^{\pm 1}])  \\
&\xrightarrow{\operatorname{ev}} \overline{\Hom(H_1(Y;\Z[t^{\pm 1}]),\Q(t)/\Z[t^{\pm 1}])}.
\end{align*}
Using the universal coefficient spectral sequence,  one can check that the evaluation map is an isomorphism, and thus so is $\Phi$.
Thus the pairing $(x,y) \mapsto \Phi(y)(x)$ is nonsingular. It is straightforward to see that this pairing is sesquilinear.  It is also Hermitian; see e.g.~\cite{PowellBlanchfield}.
\end{construction}

\begin{definition}
\label{def:Blanchfield}
Let $Y$ be a closed $3$-manifold and let $\varphi \colon \pi_1(Y) \twoheadrightarrow \Z$ be an epimorphism such that the Alexander module $H_1(Y;\Z[t^{\pm 1}])$ is torsion.
The \emph{Blanchfield form}
$$\Bl_Y \colon  H_1(Y;\Z[t^{\pm 1}]) \times H_1(Y;\Z[t^{\pm 1}]) \to \Q(t)/\Z[t^{\pm 1}]$$
is the sesquilinear, nonsingular,  Hermitian form defined by $\Bl_Y(x,y)=\Phi(y)(x)$.
\end{definition}

Let $M$ be a $4$-manifold with $\pi_1(M) = \Z$,  nondegenerate equivariant intersection form~$\lambda_M$ and ribbon boundary (meaning that the inclusion induced map $\pi_1(\partial M) \to \pi_1(M)$ is surjective).
We now outline why the symmetric boundary linking form~$\partial \lambda_M$ (that was described in Section~\ref{sub:Boundary}) is isometric to~$-\Bl_{\partial M}$.

As explained in~\cite[Remark 3.3]{ConwayPowell},  the connecting homomorphism~$\delta$ in the long exact sequence of the pair~$(M,\partial M)$, together with Poincar\'e duality and the evaluation map, determines an isomorphism
\begin{align*} D_M \colon \coker(\widehat{\lambda}_M) &\xrightarrow{\cong} H_1(\partial M;\Z[t^{\pm 1}]) \\
[x]  &\mapsto \delta \circ \PD \circ \ev^{-1}(x)
\end{align*}
that fits into the following commutative diagram:
\begin{equation}\label{eqn:delta-induces-map}
\xymatrix{
0\ar[r]&H_2(M;\Z[t^{\pm 1}]) \ar[r]^-{\widehat{\lambda}_M} \ar[d]_= & H_2(M;\Z[t^{\pm 1}])^* \ar[r] \ar[d]^{\PD \circ \ev^{-1}}_\cong & \coker(\widehat{\lambda}_M) \ar[r] \ar[d]^{D_M}_\cong & 0 \\
0\ar[r]& H_2(M;\Z[t^{\pm 1}]) \ar[r]^-{j}  & H_2(M,\partial M;\Z[t^{\pm 1}]) \ar[r]^-{\delta}  & H_1(\partial M;\Z[t^{\pm 1}]) \ar[r] & 0.  }
\end{equation}

\begin{proposition}[{\cite[Proposition 3.5]{ConwayPowell}}]
\label{prop:BoundaryLinkingForm3Manifold}
Let~$M$ be a~$4$-manifold with~$\pi_1(M)=\Z$, whose boundary is ribbon and with torsion Alexander module.  The isomorphism $D_M$ induces an isometry of symmetric linking forms
\[D_M \colon \partial (H_2(M;\Z[t^{\pm 1}] ),\lambda_M) = (\coker(\widehat{\lambda}_M),\partial \lambda_M) \xrightarrow{\cong} (H_1(\partial M;\Z[t^{\pm 1}] ),-\Bl_{\partial M}).\]
\end{proposition}

\begin{construction}
\label{cons:QuadraticRefinementOfBlanchfield}
Suppose $M$ is a spin~$4$-manifold with~$\pi_1(M) = \Z$,  whose boundary is ribbon and with torsion Alexander module.  Since $M$ is spin, the equivariant intersection form of $M$ is even. Then $\Bl_{\partial M}$ admits a preferred quadratic refinement
\begin{align*}
\mu_{\Bl_{\partial M}}  \colon H_1(\partial M;\Z[t^{\pm 1}]) &\to Q_1(\Z[t^{\pm 1}],S) \\
x &\mapsto \mu_{\partial}(D_M^{-1}(x)).
\end{align*}
Here recall from Definition~\ref{def:BoundaryLinkingForm} that $\mu_\partial$ refers to the quadratic refinement of the symmetric linking form $\partial \lambda_M$; it exists because $\lambda_M$ is even.

By construction $D_M \colon (\coker(\widehat{\lambda}_M),\partial \lambda_M,\mu_{\partial}) \xrightarrow{\cong} (H_1(\partial M;\Z[t^{\pm 1}] ),-\Bl_{\partial M},-\mu_{\Bl_{\partial M}} )$ is an isometry of quadratic linking forms.
\end{construction}

%

\subsection{Homotopy equivalences and isometries of the Blanchfield form}
\label{sub:HomotopyEquivalencesIsometry}

\purple{
Given $3$-manifolds $Y_0,Y_1$ equipped with epimorphisms $\varphi_i \colon \pi_1(Y_i) \twoheadrightarrow \Z$,  we recall when homotopy equivalences of $3$-manifolds induce isometries of the corresponding Blanchfield forms.
We then apply these considerations to boundaries of $4$-manifolds with $\pi_1=\Z$.}

\begin{proposition}[{\cite[Proposition 3.7]{ConwayPowell}}]
\label{prop:HomotopyEquivalence}
Let $Y_0,Y_1$ be $3$-manifolds equipped with epimorphisms $\varphi_i \colon \pi_1(Y_i) \twoheadrightarrow \Z$ and assume that the resulting Alexander modules are torsion for $i=0,1$.
If an orientation-preserving homotopy equivalence $f \colon Y_0 \to Y_1$ satisfies~$\varphi_1 \circ f_*=\varphi_0$ on $\pi_1(Y_0)$, then it induces an isometry between the Blanchfield forms
$$ \widetilde{f}_* \colon H_1(Y_0;\Z[t^{\pm 1}]) \to H_1(Y_1;\Z[t^{\pm 1}]).$$
\end{proposition}

The proof of Proposition~\ref{prop:HomotopyEquivalence} is fairly straightforward: the condition that $\varphi_1 \circ f_*=\varphi_0$ ensures that $f$ lifts to the infinite cyclic covers,  and the required isometry is then obtained by taking the induced map on $H_1(-;\Z[t^{\pm 1}])$.

\begin{remark}
\label{rem:HomotopyEquivalences}
Let~$M$ and $N$ be~$4$-manifolds with fundamental group $\Z$, whose boundaries are ribbon and with torsion Alexander modules.
\begin{itemize}
\item
A consequence of Proposition~\ref{prop:HomotopyEquivalence} is that
an orientation-preserving homotopy equivalence $f \colon \partial M \to \partial N$ that intertwines the 
epimorphisms $\pi_1(\partial M) \twoheadrightarrow \pi_1(M)$ and $\pi_1(\partial N) \twoheadrightarrow \pi_1(N)$
induces an isometry $\widetilde{f}_* \colon \Bl_{\partial M} \cong \Bl_{\partial N}$.
However, in general~$\widetilde{f}_*$ need not preserve the boundary quadratic refinements.
\item Consider the group~$\hAut_{\varphi}^+(\partial M)$ of orientation-preserving homotopy equivalences $f \colon \partial M \to \partial M$ that satisfy~$\varphi \circ f_*=f_* \colon \pi_1(\partial M) \to \Z$.
We  write
$$\hAut_{\varphi}^{+,\quadr}(\partial M) \subseteq \hAut_{\varphi}^+(\partial M)$$
 for the subset consisting of homotopy equivalences such that~$\widetilde{f}_*$ preserves $\mu_{\Bl_{\partial M}}$.
Proposition~\ref{prop:HomotopyEquivalence} implies that $\hAut_{\varphi}^+(\partial M)$ acts on $\Aut(\Bl_{\partial M})$ and that $\hAut_{\varphi}^{+,\quadr}(\partial M) $ acts on~$\Aut(\Bl_{\partial M},\mu_{\Bl_{\partial M}}) \subseteq \Aut(\Bl_{\partial M}).$
\item In fact~$\Aut(\Bl_{\partial M})$  and $\Aut(\Bl_{\partial M},\mu_{\Bl_{\partial M}})$ also
admit actions of $\Aut(\lambda_M)$.
In more detail, one uses $D_M$ to transport the action on $\Aut(\partial \lambda_M)$ (recall Remark~\ref{rem:ActionIsometry}) to an action on $ \Aut(\Bl_{\partial M})$: the action of~$F$ on~$h \in \Aut(\Bl_{\partial M})$ is by~$F \cdot h:=h \circ (D_M \circ \partial F \circ D_M^{-1})$.
Since $\partial F$ also preserves $\mu_{\partial}$ it follows that $\Aut(\partial \lambda_M)$ also acts on $\Aut(\Bl_{\partial M},\mu_{\Bl_{\partial M}}) \subseteq \Aut(\Bl_{\partial M})$.

%

\item Remark~\ref{remark:inclusion-of-automorphisms-quad-symm} leads to an example where
$\mu_{\Bl_{\partial M}}$, the quadratic refinement of the Blanchfield form, depends on the choice of coboundary $M$.
Consider the $4$-manifold $M:= (S^2 \widetilde{\times}_8 D^2) \natural (S^1 \times D^3)$, where
$S^2 \widetilde{\times}_8 D^2$ denotes the total space of the $D^2$-bundle over $S^2$ with Euler number $8$. Multiplication by 3 induces an automorphism of the symmetric linking form of $\partial M = L(8,1) \# (S^1 \times S^2)$ that is not induced by $\Aut(\lambda_M)$ nor by $\Homeo^+_{\varphi}(\partial M)$. Hence \cite{ConwayPiccirilloPowell} implies there is a non-homeomorphic 4-manifold $M'$ with the same boundary and intersection form as $M$. The quadratic refinements of $\Bl_{\partial M}$ induced by $M$ and $M'$ do not even lie in the same $\Homeo^+_{\varphi}(\partial M)$ orbits, and so depend strongly on the choice of coboundary.
\item It is an interesting question whether it is possible to define the quadratic refinement on $\Bl_{\partial M}$  intrinsically, using only the spin structure $M$ induces on $\partial M$. An analogue of this result is known for the standard linking form on $\partial M$; see~\cite[Sections 2.4 \& 2.5]{Deloup-Massuyeau}.
We leave this problem for later work.
\end{itemize}
\end{remark}

\subsection{Unions of spin $4$-manifolds}
\label{sub:unions-of-spin}

We describe how the union construction from Section~\ref{sub:Unions} and quadratic isometries of the Blanchfield form can be used to ensure that certain unions of spin $4$-manifolds with fundamental group $\Z$ remain spin.
\medbreak

In order to cut down on notation, we identify $(H_1(\partial M;\Z[t^{\pm 1}]),-\Bl_{\partial M},-\mu_{\Bl_{\partial M}})$ with $(\coker(\widehat{\lambda}_M),\partial \lambda_M,\mu_{\partial})$ i.e. we temporarily omit the isometry $D_M$ mentioned in Proposition~\ref{prop:BoundaryLinkingForm3Manifold} from the notation.
In particular, given an isometry $h \colon \Bl_{\partial M} \cong \Bl_{\partial N}$, we  allow ourselves to write $\lambda_M \cup_{h} -\lambda_N$.

The next proposition recalls how the algebraic union can be used to understand the equivariant intersection form of a union of two $4$-manifolds with fundamental group $\Z$.

\begin{lemma}[{\cite[Proposition 3.9]{ConwayPowell}}]
\label{lem:AlgTopUnion}
Let $M$ and $N$ be two $4$-manifolds with fundamental group $\Z$,  nondegenerate equivariant intersection forms, and whose boundaries are ribbon.
If there is an orientation-preserving homeomorphism  $f \colon \partial M_0 \cong \partial M_1$ that interwines the inclusion-induced epimorphisms $\pi_1(\partial M_0) \twoheadrightarrow \pi_1(M_0)$ and $\pi_1(\partial M_1) \twoheadrightarrow \pi_1(M_1)$,  then there is an isometry
$$\lambda_{M_0} \cup_{\widetilde{f}_*} -\lambda_{M_1} \cong \lambda_{M_0 \cup_f -M_1}$$
\end{lemma}

Note that the  intersection form being nondegenerate implies that the Alexander modules are~$\Z[t^{\pm 1}]$-torsion, via the long exact sequences of the pairs $(M,\partial M)$ and $(N,\partial N)$ with $\Z[t^{\pm 1}]$ coefficients.

\begin{proposition}
\label{prop:CPPSpinUnion}
Let $M$ and $N$ be spin $4$-manifolds with fundamental group $\Z$, nondegenerate equivariant intersection forms,  and whose boundaries are ribbon.
If $\lambda_M \cong \lambda_N$ and $f \colon \partial M\cong \partial N$ is a homeomorphism that induces an isometry between the boundary quadratic linking forms of $\partial M$ and $\partial N$, then $M \cup_{f} -N$ is spin.
\end{proposition}

\begin{proof}
Pick an isometry $F \colon \lambda_M  \cong \lambda_N$.
Applying successively Lemma~\ref{lem:AlgTopUnion} and Lemma~\ref{lem:PropertiesUnion} we obtain
$$\lambda_{M \cup_{f} -N}
 \cong  \lambda_M  \cup_{\wt{f}_*} -\lambda_N
 \cong \lambda_N \cup_{\wt{f}_* \circ \partial F^{-1}}  -\lambda_N.
 $$
Since $\wt{f}_*$ and $\partial F$ preserve the quadratic linking forms, so does $\wt{f}_* \circ \partial F^{-1}$.

As $N$ is spin, it follows that $\lambda_N$ is even.
Lemma~\ref{lem:UnionEven} implies that so is~$\lambda_N \cup_{\wt{f}_* \circ \partial F^{-1}}  -\lambda_N$. Therefore $\lambda_{M \cup_{f} -N}$ is even.
Since  the fundamental group of $M \cup_{f} -N$ is $\Z$ (see e.g.~\cite[Proposition 3.8]{ConwayPowell}) and therefore has no $2$-torsion, this implies that $M \cup_f -N$ is spin, as required.
\end{proof}

\section{Stable homeomorphism}
\label{sec:StableHomeo}

In this section,  we collect some facts about stable homeomorphism of $4$-manifolds with fundamental group $\Z$ and nonempty boundary.
We then focus on the case of $4$-manifolds whose boundary is ribbon and has torsion Alexander module.
\medbreak

Given two $4$-manifolds $M$ and $N$ with fundamental group $\pi_1(M) = \Z = \pi_1(N)$ that are either both spin or both nonspin,  we call a homeomorphism $f \colon \partial M \xrightarrow{\cong} \partial N$ a \emph{$B$-compatible} homeomorphism if the diagram
$$
\xymatrix@R0.5cm{
\pi_1(\partial M) \ar[rr]^{f_*}\ar[d]&& \pi_1(\partial N)\ar[d] \\
\pi_1(M)\ar[r]^= & \Z & \pi_1( N)\ar[l]_= \\
}
$$
commutes, and if,  in the case that $M$ and $N$ are spin, the union $M \cup_f -N$ is spin.
We write $\partial M \cong_B \partial N$ if such a homeomorphism exists, and say that $\partial M$ and $\partial N$ are \emph{$B$-homeomorphic}.

The terminology `$B$-compatible' and `$B$-homeomorphic' is motivated by our use of modified surgery below, where $B$ is shorthand for the normal 1-type of the manifolds involved.

\begin{lemma}
\label{lem:Kreck}
Let $M$ and $N$ be two $4$-manifolds with fundamental group $\Z$ and nonempty $B$-homeomorphic boundaries.  Suppose that $M$ and $N$ have the same Kirby-Siebenmann invariant and isometric equivariant intersection forms. Then $M$ and $N$ are stably homeomorphic.

In particular,  there is an equality
\[  \mathcal{S}^{\st}_{h^+,\lambda}(M)= \frac{\{N  \mid \partial N \cong_{\purple{B}} \partial M, \pi_1(N) = \Z,  \lambda_N \cong \lambda_M, \ks(N)=\ks(M) \}}{\text{o.p.~homotopy eq. of pairs}}.\]
\end{lemma}
\begin{proof}
Work of Kreck~\cite[Theorem 2]{KreckSurgeryAndDuality} ensures that $M$ and $N$ are stably homeomorphic if and only if there is a $B$-compatible homeomorphism $f \colon \partial M \xrightarrow {\cong}\partial N$ such that the union~$M \cup_f-N$ vanishes in the bordism group $\Omega_4^{\operatorname{STOP}}(S^1)\cong \Z \oplus \Z_2$ (when $M$ and $N$ are non-spin) or $\Omega_4^{\operatorname{TOPSpin}}(S^1)\cong \Z$ (when $M$ and $N$ are spin).
In the non-spin case, this bordism group is detected by the signature and Kirby-Siebenmann invariant whereas in the spin case, it is detected by the signature alone.
Since $\lambda_M \cong \lambda_N$,  we deduce that $M$ and $N$ have the same signature and Wall additivity implies that the union~$M \cup_f-N$ has vanishing signature.
The fact that $M \cup_f-N$ has vanishing Kirby-Siebenmann invariant follows from the additivity of this invariant; see e.g.~\cite[Theorem 8.2]{FriedlNagelOrsonPowell}.
\end{proof}

Given a $4$-manifold $M$ with $\pi_1(M) = \Z$,  nondegenerate equivariant intersection form and ribbon boundary,  as we will describe below,  the methods of~\cite{ConwayPiccirilloPowell} produce $4$-manifolds $N$ with $\pi_1(N) = \Z$, ribbon boundary, $\lambda_M \cong \lambda_N$, $\ks(M)=\ks(N)$ and a preferred identification $g \colon \partial M \cong \partial N$.
Since $\lambda_M \cong \lambda_N$ and $\pi_1(M)= \Z =\pi_1(N)$ has no $2$-torsion, either~$M$ and~$N$ are both spin or both nonspin.
By Lemma~\ref{lem:Kreck}, in order to prove that~$M$ and $N$ are stably homeomorphic,  it therefore suffices to know that in the spin case, the union~$M \cup_{g} -N$ is spin, for which we use Proposition~\ref{prop:CPPSpinUnion}.

\begin{proposition}
\label{prop:CPPStablyHomeo}
Let $M$ and $N$ be spin $4$-manifolds with fundamental group $\Z$, nondegenerate equivariant intersection forms,  and whose boundaries are ribbon.
If $\lambda_M \cong \lambda_N$, and $g \colon \partial M\cong \partial N$ is a homeomorphism that induces an isometry between the boundary quadratic linking forms of $\partial M$ and $\partial N$,
then $M$ and $N$ are stably homeomorphic.
\end{proposition}

\begin{proof}
Proposition~\ref{prop:CPPSpinUnion} implies that the homeomorphism $\partial M \cong \partial N$ is $B$-compatible and the result therefore follows from  Lemma~\ref{lem:Kreck}.
\end{proof}


\color{black}

\section{From stable homeomorphism to isometries of the Blanchfield form}
\label{sec:Reduc}

In this section,~$M$ denotes a \purple{spin}~$4$-manifold with an identification~$\pi_1(M) = \Z$,  ribbon boundary (the inclusion induced map~$\varphi \colon \pi_1(\partial M) \to \pi_1(M)$ is surjective), and nondegenerate equivariant intersection form~$\lambda_M$.
\purple{As we recalled in Section~\ref{sec:Blanchfield},}
since~$\lambda_M$ is nondegenerate,  the Alexander module~$H_1(\partial M;\Z[t^{\pm 1}])$ is torsion and supports the Blanchfield form,
$$\Bl_{\partial M} \colon H_1(\partial M;\Z[t^{\pm 1}]) \times H_1(\partial M;\Z[t^{\pm 1}]) \to \Q(t)/\Z[t^{\pm 1}],$$
which is nonsingular, sesquilinear, and Hermitian.
The goal of this section is to describe in more detail the set~$\Aut(\Bl_{\partial M},\purple{\mu_{\Bl_{\partial M}}})/(\Aut(\lambda_M) \times \hAut_{\varphi}^{+,\purple{\quadr}}(\partial M))$ that was mentioned in the introduction
and then to prove Proposition~\ref{prop:SurjectionOntoAlgebraIntro}.
\medbreak
We start by
\purple{recalling} the aforementioned actions of~$\hAut_{\varphi}^{+,\purple{\quadr}}(\partial M)$ and~$\Aut(\lambda_M)$ on the group~$\Aut(\Bl_{\partial M},\purple{\mu_{\Bl_{\partial M}}})$ of isometries of the \purple{quadratic refinement of the} Blanchfield form of~$\partial M$ \purple{induced by the even form $\lambda_M$.}
Here recall that~$\hAut_{\varphi}^{+,\purple{\quadr}}(\partial M)$ denotes the group of orientation-preserving homotopy equivalences~$f \colon \partial M \to \partial M$ that satisfy~$\varphi \circ f_*=f_* \colon \pi_1(\partial M) \to \Z$ \purple{and preserve $\purple{\mu_{\Bl_{\partial M}}}$}, while~$\Aut(\lambda_M)$ denotes the set of isometries of~$\lambda_M$.
\begin{itemize}
\item We
\purple{recall} the action of~$\hAut_{\varphi}^{+,\purple{\quadr}}(\partial M)$ on~$\Aut(\Bl_{\partial M},\purple{\mu_{\Bl_{\partial M}}})$.
Since any homotopy equivalence~$f \in \hAut_{\varphi}^+(\partial M)$ satisfies~$\varphi \circ f_*=f_*$, it lifts to a homotopy equivalence~$\widetilde{f}$ on the~$\Z$-covers that induces a~$\Z[t^{\pm 1}]$-linear map on homology; we denote this map by~$\widetilde{f}_*$.
Since~$f$ is orientation-preserving, so is~$\widetilde{f}$ and it follows that~$\widetilde{f}_*$ is an isometry of the Blanchfield form.
The action of~$f$ on~$h \in \Aut(\Bl_{\partial M},\purple{\mu_{\Bl_{\partial M}}})$ is then by~$f \cdot h=\widetilde{f}_* \circ h$.
\item We describe the action of~$\Aut(\lambda_M)$ on~$\Aut(\Bl_{\partial M},\purple{\mu_{\Bl_{\partial M}}})$.
\purple{Recall from Section~\ref{sec:Blanchfield} that}
the \purple{even} Hermitian form~$\lambda_M$ determines an adjoint map $\widehat{\lambda}_M \colon H_2(M;\Z[t^{\pm 1}]) \to H_2(M;\Z[t^{\pm 1}])^*$, and a~$\Q(t)/\Z[t^{\pm 1}]$-valued \purple{quadratic} linking form
\begin{align}
\label{eq:BoundaryLinkingForm}
\partial \lambda_M \colon \coker(\widehat{\lambda}_M) \times \coker(\widehat{\lambda}_M) &\to \Q(t)/\Z[t^{\pm 1}];\;\; \partial \lambda_M([x],[y])=y(z)/p, \nonumber \\
\purple{\mu_{\partial} \colon \coker(\widehat{\lambda}_M)} &\to Q_+(\Z[t^{\pm 1}],S); \;\; \mu_{\partial}(\purple{[x]})=\purple{x}(z)/p.
\end{align}
Here~$p \in \Z[t^{\pm 1}]$ and $z \in H_2(M;\Z[t^{\pm 1}])$ satisfy~$px=\widehat{\lambda}_M(z)$.
\purple{In Section~\ref{sec:Blanchfield} we also recalled the definition of the isometry}
$$D_M \colon -\partial \lambda_M \cong \Bl_{\partial M},$$
\purple{and that} an isometry~$F \in \Aut(\lambda_M)$
\color{black}
induces an isometries~$\partial F \colon -\partial \lambda_M \cong -\Bl_{\partial M}$ and~$\partial F \colon \mu_{\partial} \cong  -\mu_{\Bl_{\partial M}}$
\color{black}
by noting that the isomorphism~$(F^*)^{-1}$ descends to an isometry on the cokernels.
The action of~$F$ on~$h \in \Aut(\Bl_{\partial M},\purple{\mu_{\Bl_{\partial M}}})$ is then by~$F \cdot h:=h \circ (D_M \circ \partial F \circ D_M^{-1})$.
\end{itemize}

Note that the actions commute, because one acts by pre-composition and the other acts by post-composition.
We obtain an action of the product $\Aut(\lambda_M) \times \purple{\hAut_{\varphi}^{+,\purple{\quadr}}(\partial M)}$ on $\Aut(\Bl_{\partial M},\purple{\mu_{\Bl_{\partial M}}})$.
Now that we have made sense of the orbit set $\Aut(\Bl_{\partial M},\purple{\mu_{\Bl_{\partial M}}})/(\Aut(\lambda_M) \times \hAut_{\varphi}^{+,\purple{\quadr}}(\partial M)$, we describe its relation to the homotopy stable class of~$M$.
Recall from \purple{Lemma~\ref{lem:Kreck} that}
\begin{equation}
\label{eq:DescriptionOfStable}
 \mathcal{S}^{\st}_{h^+,\lambda}(M) = \frac{ \{N  \mid \partial N \cong_{\purple{B}} \partial M,
\pi_1(N) = \Z,
\lambda_N \cong \lambda_M, \ks(N)=\ks(M) \}}{\text{o.p.\ hom.\ equiv.\ of pairs}}.\end{equation}
In order to relate~$\mathcal{S}^{\st}_{h^+,\lambda}(M)$ to the orbit set~$\Aut(\Bl_{\partial M},\purple{\mu_{\Bl_{\partial M}}})/(\Aut(\lambda_M) \times \hAut_{\varphi}^{+,\purple{\quadr}}(\partial M))$, we recall a construction from~\cite[Construction 1]{ConwayPiccirilloPowell} which has its origins in work of Boyer~\cite{BoyerRealization}.

\begin{construction}
\label{cons:CPP22}
We describe a map~$b \colon \mathcal{S}_{h^+,\lambda}^{\st}(M) \to \Aut(\Bl_{\partial M})/(\Aut(\lambda_M) \times \hAut_{\varphi}^+(\partial M))$.
Given a $4$-manifold~$N \in  \mathcal{S}^{\st}_{h^+,\lambda}(M)$,  pick a homeomorphism~$g \colon \partial N \cong \partial M$
and an isometry~$F \colon \lambda_M \cong \lambda_N$.
Since $N \in  \mathcal{S}^{\st}_{h^+,\lambda}(M)$,  it also has ribbon boundary and torsion Alexander module,  thus ensuring that the isometry $D_N \colon -\partial \lambda_N \cong \Bl_{\partial N}$ is defined.
Now set
\[b(N):=g_* \circ D_N \circ \partial F \in \Aut(\Bl_{\partial M})/(\Aut(\lambda_M) \times \hAut_{\varphi}^+(\partial M)).\]
One verifies that $b(N)$ is independent of the choices of~$F,g$ and the orientation-preserving homotopy equivalence class of~$(N,\partial N)$.
See \cite{ConwayPiccirilloPowell} for details.

\color{black}
Suppose that in addition we assume the homeomorphism $g$ is such that the induces homomorphism $g_* \colon H_1(\partial N;\Z[t^{\pm 1}]) \to H_1(\partial M;\Z[t^{\pm 1}])$ intertwines the quadratic refinements $\mu_{\Bl_{\partial N}}$ and~$\mu_{\Bl_{\partial M}}$.
\purple{To see that such a homeomorphism exists,  restrict a stable homeomorphism $\Phi \colon M \# W_r \cong N \# W_r$ to the boundaries:
this will intertwine the quadratic refinements because stabilising a Hermitian form by $\bsm 0&1 \\ 1&0 \esm$ does not affect the boundary quadratic linking form.}
Since both $D_N$ and~$\partial F$ preserve the quadratic refinements, it follows that $b$ in fact defines a map
$$b \colon \mathcal{S}_{h^+,\lambda}^{\st}(M) \to \Aut(\Bl_{\partial M},\mu_{\Bl_{\partial M}})/(\Aut(\lambda_M) \times \hAut_{\varphi}^{+,\purple{\quadr}}(\partial M)).$$
\color{black}
\end{construction}

The next proposition proves Proposition~\ref{prop:SurjectionOntoAlgebraIntro} from the introduction.

 \begin{proposition}
 \label{prop:SurjectionOntoAlgebra}
 If~$M$ is a \purple{spin}~$4$-manifold with~$\pi_1(M) = \Z$,  ribbon boundary and nondegenerate equivariant intersection form~$\lambda_M$,  then the map~$b$ from Construction~\ref{cons:CPP22} defines a surjection
$$ b\colon  \mathcal{S}_{h^+ ,\lambda}^{\st}(M) \twoheadrightarrow  \Aut(\Bl_{\partial M},\purple{\mu_{\Bl_{\partial M}}})/(\Aut(\lambda_M)\times \hAut_{\varphi}^{+,\purple{\quadr}}(\partial M)).$$
\end{proposition}
 \begin{proof}
\purple{According to}~\cite[Theorem 1.15]{ConwayPiccirilloPowell}, after fixing one choice of isometry $\partial \lambda_M \cong - \Bl_{\partial M}$ \purple{-- we will use $D_M$ --}
\purple{every element $\Aut(\Bl_{\partial M})/\Aut(\lambda_M)$}
 is realised by a~$4$-manifold $N$ with~$\pi_1(N) = \Z$, ribbon boundary $\partial N$ homeomorphic to~$\partial M$ \purple{via a homeomorphism $g \colon \partial N \cong \partial M$}, equivariant intersection form $\lambda_N$ isometric to~$\lambda_M$ \purple{via an isometry $F \colon \lambda_M \cong \lambda_N$}, and~$\ks(N)=\ks(M)$.
\color{black}
In particular we can realise every element $b \in \Aut(\Bl_{\partial M},\mu_{\Bl_{\partial M}})/(\Aut(\lambda_M)\times \hAut_{\varphi}^{+,\quadr}(\partial M))$ by such a manifold~$N$.

As indicated in~\cite[Proof of Proposition 4.14]{ConwayPiccirilloPowell},  we have $b=g_* \circ D_N \circ \partial F$.
Since~$b$, $D_N$, and $\partial F$ are isometries that preserve quadratic refinements,  so does $g$.

Proposition~\ref{prop:CPPStablyHomeo} now ensures that~$M$ and $N$ are stably homeomorphic.
We have therefore produced $N \in \mathcal{S}_{h^+ ,\lambda}^{\st}(M)$ with $b(N)=b$, as required.
 \end{proof}

\section{Infinite automorphism sets}
\label{sec:InfinitebAut}

In this section,  we conclude the proof of Theorem~\ref{thm:Infinite} by showing that \[\Aut(\Bl_{\partial M},\purple{\mu_{\Bl_{\partial M}}})/(\Aut(\lambda_M)\times \hAut_{\varphi}^{+,\purple{\quadr}}(\partial M))\] is countably infinite when \[M=X_{2q}(U) \natural (S^1 \times D^3),\]  with $q$ an odd prime.
Here~$X_{2q}(U)$ denotes the~$2q$-trace on the unknot~$U$, i.e. the smooth~$4$-manifold obtained from~$D^4$ by attaching a~$2q$-framed~$2$-handle along the unknot.
The plan is to first study~$\Aut(\Bl_{\partial M},\purple{\mu_{\Bl_{\partial M}}})/\Aut(\lambda_M)$ and then to consider the action by the self-homotopy equivalences of~$\partial M$.
\purple{In fact we will show in Lemma~\ref{lem:Quad=Sym} that $\Aut(\Bl_{\partial M}, \purple{\mu_{\Bl_{\partial M}}})/\Aut(\lambda_M)=\Aut(\Bl_{\partial M})/\Aut(\lambda_M)$. So first we will consider the latter set, ignoring quadratic refinements for a moment.}

\medbreak

To study~$\Aut(\Bl_{\partial M})/\Aut(\lambda_M)$,  recall from Section~\ref{sec:Reduc} that~$\Bl_{\partial M}$ is isometric to the linking form~$\partial \lambda_M$ defined in~\eqref{eq:BoundaryLinkingForm}, at the start of Section~\ref{sec:Reduc}.
In particular,  the isometry~$D_M \colon -\partial \lambda_M \cong \Bl_{\partial M}$ induces a bijection
$$ \Aut(\Bl_{\partial M})/\Aut(\lambda_M) \cong  \Aut(\partial \lambda_M)/\Aut(\lambda_M).$$
For rank one forms (such as~$(H_2(M;\Z[t^{\pm 1}]),\lambda_M)=(\Z[t^{\pm 1}],\lambda_{2q})$, this set admits a particularly convenient description.

Given a ring~$R$ with involution~$x \mapsto \overline{x}$,  the group of \emph{unitary units}~$U(R)$ refers to those~$u \in R$ such that~$u \overline{u}=1$, with the group operation given by restricting the multiplication on $R$.
For example, when~$R=\Z[t^{\pm 1}]$, all units are unitary and are of the form~$\pm t^{k}$ with~$k \in \Z$.
In what follows, we make no distinction between rank one Hermitian forms and symmetric Laurent polynomials.
The next lemma follows by unwinding the definition of~$\Aut(\partial \lambda)$; see also~\cite[Remark 1.16]{ConwayPowell} and~\cite[Lemma~7.1]{ConwayPiccirilloPowell}.

\begin{lemma}
\label{lem:Rank1}
If~$\lambda \in \Z[t^{\pm 1}]$ is a symmetric Laurent polynomial, then
$$\Aut(\partial \lambda)/\Aut(\lambda)=U(\Z[t^{\pm 1}]/\lambda)/U(\Z[t^{\pm 1}]).$$
\end{lemma}

\begin{proposition}
\label{ex:UnitaryUnitsZZ}
Given an odd prime $q$, the map
\begin{align*}
\Theta \colon \Z &\xrightarrow{\cong} U(\Z[t^{\pm 1}]/2q)/U(\Z[t^{\pm 1}]) \\
n &\mapsto (q{-}1)t^n+q
\end{align*}
is a group isomorphism.
\end{proposition}

\begin{proof}
One verifies that $(q{-}1)t^n+q$ is a unitary unit by using that~$q(q{-}1) \equiv 0$ mod $2q$ (recall that~$q$ is odd).
We then check that~$\Theta$ is a homomorphism:
\begin{align*}
n+m \mapsto ((q{-}1)t^n+q)((q{-}1)t^m+q)
&= (q{-}1)^2t^{n+m} + q(q{-}1) (t^m +t^n) + q^2  \\
& \sim -(q{-}1)t^{n+m} - q \\
& \sim (q{-}1)t^{n+m} +q.
\end{align*}
Here the penultimate equivalence uses that $(q{-}1)^2=q(q{-}1)-(q{-}1) \equiv -(q{-}1)$ mod $2q$ and $q^2 \equiv q \equiv -q$ mod $2q$.
 The last equivalence uses that~$-1 \in U(\Z[t^{\pm 1}])$.

Next we show that~$\Theta$ is injective.
If~$(q{-}1)t^n + q$ were trivial, we would have~$(q{-}1)t^n + q = \pm t^k \in \Z[t^{\pm 1}]/2q$ for some~$k$, but this is true only if~$n=0$.

Now we show that~$\Theta$ is surjective.
An explicit verification shows that the following map is an isomorphism:
\begin{align*}
U(\Z[t^{\pm 1}]/2) \times U(\Z[t^{\pm 1}]/q) & \to U(\Z[t^{\pm 1}]/2q) \\
(a,b) &\mapsto qa-(q{-}1)b.
\end{align*}
To see this one should check that~$(qa-(q{-}1)b)(q\overline{a}-(q{-}1)\overline{b}) \equiv 1$ when~$a\bar{a}=1= b\bar{b}$, which implies that the map lands in the claimed target.
The inverse is given by ~$x \mapsto ([x]_2,[x]_{q})$, i.e.\ considering the coefficients modulo $2$ and $q$ respectively.
Checking that this is the inverse homomorphism implies that the map is an isomorphism as asserted.

The units of~$\Z[t^{\pm 1}]/2$ are of the form~$t^m$ for~$m \in \Z$.
On the other hand, since $q$ is an odd prime, the unitary units of $\Z[t^{\pm 1}]/q$ are of the form~$\pm t^n$ for~$n \in \Z$.
It follows that
$$U(\Z[t^{\pm 1}]/2q) \cong \{qt^m+ (q{-}1) \varepsilon t^n \mid n,m\in\Z,
\varepsilon\in\{\pm 1\}\}.$$
Passing to the quotient by~$U(\Z[t^{\pm 1}])$ yields the required isomorphism, because once we can multiply by~$\pm t^k$ for any~$k \in \Z$, we have
$qt^m + (q{-}1)\varepsilon t^n  \sim (q{-}1) \varepsilon t^{n-m}+q$.
Also \[-(q{-}1)t^{n-m}+q \sim -(q{-}1)t^{n-m}-q \sim (q{-}1)t^{n-m}+q,\]
so we can ignore the~$\varepsilon$, and every element of~$U(\Z[t^{\pm 1}]/2q)$ is of the form~$(q{-}1)t^k +q$ for some~$k \in \Z$.  So~$\Theta$ is indeed surjective, which completes the proof that~$\Theta$ is an isomorphism.
\end{proof}

\color{black}
\begin{lemma}
\label{lem:Quad=Sym}
Given an odd prime $q \in \Z$, for the Hermitian form~$\lambda=2q \in \Z[t^{\pm 1}]$, one has
$$\Aut(\partial \lambda, \purple{\mu_\partial})/\Aut(\lambda)=\Aut(\partial \lambda)/\Aut(\lambda).$$
\end{lemma}

\begin{proof}
The inclusion $\Aut(\partial \lambda, \purple{\mu_\partial})/\Aut(\lambda) \subseteq \Aut(\partial \lambda)/\Aut(\lambda)$ always holds. So, thanks to Proposition~\ref{ex:UnitaryUnitsZZ},  the lemma reduces to proving that for every $n \in \Z$ multiplication by $(q-1)t^n+q$ as a map $\Z[t^{\pm 1}]/2q \to \Z[t^{\pm 1}]/2q$ preserves the quadratic refinement $\mu_\partial(x)=\smfrac{1}{2q} x$.
Writing $q=2k+1$, a direct calculation in $Q_1(\Z[t^{\pm 1}],S)$ now shows that
\begin{align*}
\smfrac{1}{2q}((q-1)t^n+q)((q-1)t^{-n}+q)
&=\smfrac{1}{2q}((q-1)^2+q^2)+\smfrac{1}{2q}(q(q-1)(t^n+t^{-n})) \\
&=\smfrac{1}{2q} + 2k+ k(t^n+t^{-n}) \equiv \smfrac{1}{2q}.
\end{align*}
This concludes the proof of the lemma.
\end{proof}
\color{black}

For $q$ an odd prime, the combination of Lemma~\ref{lem:Rank1}, Proposition~\ref{ex:UnitaryUnitsZZ}, \purple{and Lemma~\ref{lem:Quad=Sym}} implies that for~$M=X_{2q}(U) \natural (S^1 \times D^3)$ we have
$$\purple{\Aut(\Bl_{\partial M},\purple{\mu_{\Bl_{\partial M}}})/\Aut(\lambda_M)
= \Aut(\Bl_{\partial M})/\Aut(\lambda_M)
\cong \Z.}$$
We now study the effect of factoring out by~$\hAut^{+,\purple{\quadr}}_\varphi(\partial M)$.

\color{black}
\begin{remark}
\label{rem:DontHaveToWorkTooHard}
Our strategy will be to determine $\hAut^+_\varphi(\partial M)$ and show that the effect of its action on $ \Aut(\Bl_{\partial M})/\Aut(\lambda_M)$ is given by multiplication by elements of the form $\pm t^n$, $n \in \Z$.
This will automatically imply that $\hAut^{+,\quadr}_\varphi(\partial M)$ also acts on $\Aut(\Bl_{\partial M},\mu_{\Bl_{\partial M}})/\Aut(\lambda_M)$ by multiplication by elements of the form $\pm t^n$.
\end{remark}
\color{black}

In fact we will make this argument in a slightly more general setting.
Consider the~$3$-manifold~$Y:=N \# (S^1 \times S^2)$, where~$N$ is a 3-manifold with finite fundamental group.
We fix an identification~$H_1(S^1 \times S^2)= \Z$, an identification $\pi_1(Y) = \pi_1(N) \ast \Z$, and consider the finite abelian group~$A:=TH_1(Y) \cong H_1(N)$.
Let~$\varphi \colon \pi_1(Y) \to H_1(Y)/TH_1(Y)=\Z$ be the canonical projection onto the free part of~$H_1(Y)$.
In what follows, to distinguish~$H_1(S^1 \times S^2)=\Z$ from the free~$\Z$-factor of~$\pi_1(Y) \cong \pi_1(N) \ast \Z$,  we will exclusively write~$H_1(S^1 \times S^2)$ as~$\langle t \rangle$.

Summarising the notation, we have
\begin{align*}
A := TH_1(Y) \cong H_1(N),\,\,
\varphi  \colon \pi_1(Y) \twoheadrightarrow \langle t \rangle, \text{ and }
\theta \colon \pi_1(N) \stackrel{\operatorname{ab}}{\twoheadrightarrow} H_1(\pi_1(N))= H_1(N) = A.
\end{align*}
The example we have in mind is~$Y_q=\partial M=L(2q,1) \# (S^1 \times S^2)$, where $M = X_{2q}(U) \natural S^1 \times D^3$, so that~$A \cong \Z/2q$ and~$\varphi \colon \pi_1(Y_q) \twoheadrightarrow \langle t \rangle$ coincides with the inclusion induced map~$\pi_1(\partial M) \twoheadrightarrow \pi_1(M)=\Z$.

Returning to the more general setting where~$Y=N \# (S^1 \times S^2)$ with~$N$ a 3-manifold with $\pi_1(N)$ finite, the epimorphism~$\varphi \colon \pi_1(Y) \twoheadrightarrow  \langle t \rangle$ induces an infinite cyclic cover~$Y^\infty$ with
$$H_1(Y^\infty) \cong H_1(Y;\Z[t^{\pm 1}]) \cong H_1(\ker(\varphi)).$$
Our goal is now to describe the isomorphism type of this~$\Z[t^{\pm 1}]$-module (this is the content of Construction~\ref{cons:Isomorphism} and Lemma~\ref{lem:IsomorphismTH} below) and to then deduce the effect of the action of~$\hAut_{\varphi}^+(Y)$ on~$H_1(Y;\Z[t^{\pm 1}])$ in Proposition~\ref{prop:im(H_*)small}.

In what follows, we write~$A[t^{\pm 1}]$ for the abelian group of Laurent polynomials with coefficients in the finite abelian group~$A$.

\begin{construction}
\label{cons:Isomorphism}
We construct a group homomorphism~$\Psi \colon A[t^{\pm 1}] \to H_1(\ker(\varphi))$.

Elements of~$A[t^{\pm 1}]$ are of the form~$\sum_i  a_it^i$ with ~$a_i \in A$.
As the map~$\varphi \colon \pi_1(N) \ast \Z \to \Z$ is surjective, we can write each~$t^i$ as~$\varphi(g_i) =t^i$ for some~$g_i \in \pi_1(N) \ast \Z$.
The abelianisation~$\theta \colon \pi_1(N) \to A=H_1(\pi_1(N))$ is also surjective, so we can write each~$a \in A$ as~$a=\theta(p)$ for some~$p\in \pi_1(N)$.
 We can therefore write an element of~$A[t^{\pm 1}]$ as~$\sum_i \theta(p_i) \varphi(g_i)$.
Since~$p_i \in \pi_1(N),g_i \in \pi_1(N) \ast \Z$ and~$A\subseteq \ker(\varphi)$, we can consider the element ~$g_ip_ig_i^{-1}$ as an element of~$\ker(\varphi) \subseteq \pi_1(N) \ast \Z$ and use~$[g_ip_ig_i^{-1}] \in H_1(\ker(\varphi))$ to denote its image in the abelianisation.
 Define the map~$\Psi$ as
  \begin{align*}
   \Psi \colon &A[t^{\pm 1}]  \to  H_1(\ker(\varphi)) \\ &\sum_i a_it^i=\sum_i \theta(p_i) \varphi(g_i) \mapsto  \sum_i [g_i p_i g_i^{-1}].
   \end{align*}
We show that~$\Psi$ does not depend on the choice of the~$p_i$ and the~$g_i$.
First we argue that the definition of~$\Psi$ does not depend on the choice of the~$p_i$.
It suffices to show that if~$\theta(p)=\theta(p')$, then~$\Psi(\theta(p)\varphi(g))=\Psi(\theta(p')\varphi(g))$ for every~$g \in \pi_1(N) \ast \Z$.
Since~$\theta(p(p')^{-1})=0$, we know that~$p{p'}^{-1}$ lies in the commutator subgroup~$\pi_1(N)^{(1)} = [\pi_1(N),\pi_1(N)].$  Therefore,  since $\pi_1(N)^{(1)}$ is normal, $gp{p'}^{-1}g^{-1} = (gpg^{-1})(g{p'}^{-1} g^{-1}) \in \pi_1(N)^{(1)}$ for all~$g \in \pi_1(N) \ast \Z$. Since~$\pi_1(N) \subseteq \ker (\varphi)$, it follows that~$\pi_1(N)^{(1)} \subseteq (\ker (\varphi))^{(1)}$, and therefore~$(gpg^{-1})(g{p'}^{-1} g^{-1}) \in (\ker (\varphi))^{(1)}$, from which it follows~$(gpg^{-1})(g{p'}^{-1} g^{-1})  =0 \in H_1(\ker ( \varphi))$.
We deduce that~$[g p g^{-1}]=[g p'g^{-1}] \in H_1(\ker(\varphi))$ and thus
$$\Psi(\theta(p)\varphi(g))=[g p g^{-1}]=[g p'g^{-1}]=\Psi(\theta(p')\varphi(g)) \in H_1(\ker(\varphi)).$$
This proves that~$\Psi$ does not depend on the choice of the~$p_i$.

Next, we argue that the definition of~$\Psi$ does not depend on the choice of the~$g_i$.
This time,  it suffices to prove that if~$\varphi(g) = \varphi({g'})$ and~$p \in \pi_1(N)$, then~$\Psi(\theta(p)\varphi(g))=~\Psi(\theta(p)\varphi({g'}))$.
This latter equality holds if and only if
$[gpg^{-1} {g'}p^{-1}{g'}^{-1}] = 0 \in H_1(\ker(\varphi))$, which in turn, by conjugating with $g^{-1}$, holds if and only if~$[pg^{-1} {g'}p^{-1}{g'}^{-1}g]=0 \in H_1(\ker(\varphi))$.
But since~$pg^{-1} {g'}p^{-1}{g'}^{-1}g$ is a commutator of~$p$ and~$g^{-1}g'$, which both lie in~$\ker(\varphi)$,  we indeed obtain~$[pg^{-1} {g'}p^{-1}{g'}^{-1}g]=0 \in H_1(\ker(\varphi))$.

This concludes the verification that~$\Psi$ does not depend on any of the choices we made.
One also verifies readily that~$\Psi$ is a group homomorphism. This completes Construction~\ref{cons:Isomorphism}.
\end{construction}

As in Construction~\ref{cons:Isomorphism}, for each~$h \in \langle t \rangle$, we fix a~$g \in \pi_1(N) \ast \Z$ such that~$\varphi(g)=h$.
This choice will be used again in the next lemma which establishes that the map~$\Psi$ is an isomorphism.

\begin{lemma}
\label{lem:IsomorphismTH}
The map~$\Psi \colon A[t^{\pm 1}]  \to  H_1(\ker(\varphi))$  from Construction~\ref{cons:Isomorphism} is an isomorphism.
\end{lemma}

\begin{proof}
We construct an inverse~$\Theta \colon H_1(\ker(\varphi)) \to A[t^{\pm 1}]$ to~$\Psi$.
A word~$w \in \ker(\varphi) \subseteq \pi_1(N) \ast \Z$ representing an element of~$H_1(\ker(\varphi))$ is a product of elements of~$\pi_1(N)$ and~$\Z$.

By introducing cancelling pairs of the type~$g_i^{-1}g_i$ in between each occurrence of a~$p'_k \in \pi_1(N)$ in~$w$, we can arrange that for some elements~$\widetilde{g}_k \in \pi_1(N) \ast \Z$ and~$p'_k \in \pi_1(N)$, the word~$w$ is of the form
$$w=\prod_k \widetilde{g}_k p'_k \widetilde{g}_k^{-1}.$$
Here it is crucial to use that~$w \in \ker(\varphi)$. For example if~$w=p_1'n_1p_2'n_2p_3'n_3$,  for $p_i \in \pi_1(N)$ and $n_j \in \Z$, then since~$w \in \ker(\varphi)$ we know that~$n_3=(n_1n_2)^{-1}=(n_2n_1)^{-1}$.
Therefore we can express~$w$ as~$w=p_1'n_1p_2'n_1^{-1}(n_1n_2)p_3'(n_1n_2)^{-1}$.

As was mentioned before the lemma, 
we fixed a preferred~$g_j \in \pi_1(N) \ast \Z$ with~$\varphi(g_j)=\varphi(\widetilde{g}_k)$.
Arguing as in Construction~\ref{cons:Isomorphism} (when we showed that the choice of the~$g_i$ is immaterial), up to commutators in~$[\ker(\varphi),\ker(\varphi)]$, we can replace ~$\widetilde{g}_k p_k' \widetilde{g}_k^{-1}$ with~$g_j p_k' g_j^{-1}$.
Next, working in~$H_1(\ker(\varphi)) = \ker(\varphi)_{ab}$ and collecting terms with the same conjugating element~$g_j$, we obtain an element of the form~$\sum_j [g_j p_j g_j^{-1}]$, where~$p_j = \prod_{\lbrace k \mid \varphi(\widetilde{g}_k) = \varphi(g_j) \rbrace } p_k'$.
We can therefore define a map
\begin{align*}
\Theta \colon H_1(\ker(\varphi)) & \to A[t^{\pm 1}] \\
 {[w]} &\mapsto  \sum_j \theta(p_j) \varphi(g_j).
 \end{align*}
One checks that the map~$\Theta$ is a homomorphism and is the inverse to~$\Psi$. Thus~$\Psi$ is an isomorphism.
\end{proof}

We are now able to describe the action of~$\hAut_{\varphi}^+(Y)$ on~$H_1(Y;\Z[t^{\pm 1}])$.

\begin{proposition} \label{prop:im(H_*)small}
Let $at^{\ell}\in H_1(Y;\Z[t^{\pm 1}]) \cong A[t^{\pm 1}]$.
The action of~$f \in \hAut_{\varphi}^+(Y)$ sends $at^{\ell} \mapsto a'  t^{k + \ell}$, for some~$k \in \Z$ and for some element~$a' \in A$ having the same order as~$a$.
\end{proposition}

\begin{proof}
As in Construction~\ref{cons:Isomorphism},  we can represent any element of~$A[t^{\pm 1}]$ as a sum of~$\theta(p)\varphi(g)$, where~$p \in \pi_1(N)$ and~$g \in \pi_1(N) \ast \Z$.
We will describe~$f_*(\theta(p)\varphi(g))$.

In fact,  since we have the following commutative diagram of isomorphisms
\[\xymatrix{A[t^{\pm 1}] \ar[r]_-{\Psi}^-{\cong} \ar[d]_{f_*}^{\cong} & H_1(\ker(\varphi)) \ar[d]_{f_*}^{\cong} \\
A[t^{\pm 1}] \ar[r]_-{\Psi}^-{\cong} & H_1(\ker(\varphi)),}\]
it is equivalent to describe~$\Psi^{-1} \circ f_* \circ \Psi (\theta(p)\varphi(g))$.
First, the definition of~$\Psi$ implies that $\Psi (\theta(p)\varphi(g))=[gpg^{-1}] \in H_1(\ker (\varphi))$.
Applying~$f_*$, we then obtain~$[f_*(g)f_*(p)f_*(g)^{-1}] \in H_1(\ker (\varphi))$.

But now, under an isomorphism~$\pi_1(N) \ast \Z \xrightarrow{\cong} \pi_1(N) \ast \Z$, every element of~$\pi_1(N)$ is sent to an element of finite order, since~$\pi_1(N)$ is finite.
This implies that for every~$p \in \pi_1(N)$, we have that~$f_*(p) = hp'h^{-1} \in \pi_1(N) \ast \Z$ for some~$p' \in \pi_1(N)$ and some~$h \in \pi_1(N) \ast \Z$. This follows by considering the cyclic subgroup generated by~$f_*(p)$ and applying the Kurosh subgroup theorem, which implies that a finite subgroup of a free product of nontrivial groups is a conjugate of a finite subgroup of one of the factors.

Next, since~$f_*$ is an isomorphism, ~$[f_*(p)] = [hp'h^{-1}]$ has the same order as~$[p]$ in $H_1(\ker(\varphi))$.
Since they are conjugate, in $\pi_1(N) \ast \Z$, we know that $hp'h^{-1}$ and $p'$ have the same order.
We claim that $[hp'h^{-1}]$ has the same order as~$[p']$ in $H_1(\ker (\varphi))$.

To prove the claim, suppose that $\ord([p']) = k$. Then $[(p')^k] = 0 \in H_1(\ker (\varphi))$, i.e.\ $(p')^k \in \ker (\varphi)^{(1)}$.  Since $\ker (\varphi)$ is normal, for every $x \in \ker (\varphi)$ we have that $hxh^{-1} \in \ker (\varphi)$, and therefore since $h[x,y]h^{-1} = [hxh^{-1},hyh^{-1}]$,  for every $z \in \ker (\varphi)^{(1)}$ we have that $hzh^{-1} \in \ker (\varphi)^{(1)}$.  Thus $h(p')^kh^{-1} = (hp'h^{-1})^k \in \ker(\varphi)^{(1)}$, and therefore $\ord([hp'h^{-1}]) \leq k = \ord([p'])$.  Since $p'$ is also a conjugate of  $hp'h^{-1}$, by symmetry we also have $\ord([p']) \leq \ord([hp'h^{-1}])$, and so we have equality. This completes the proof of the claim.

The claim implies that in $H_1(\ker (\varphi))$ we have \[\ord ([p']) = \ord([hp'h^{-1}]) = \ord([f_*(p)]) = \ord ([p]).\]
Returning to the main arc of the proof, so far we have
\[
f_* \circ \Psi (\theta(p)\varphi(g))=[f_*(g)f_*(p)f_*(g)^{-1}] = [f_*(g)hp'h^{-1}f_*(g)^{-1}]
\]
and it remains to apply~$\Psi^{-1}$.
The effect of $\Psi^{-1}$ is~$\theta(p')\varphi(h)\varphi(f_*(g)) \in A[t^{\pm 1}]$.
Since~$f \in \hAut_{\varphi}^+(Y)$, we have~$\varphi \circ f_*=\varphi$ and therefore
\[f_*(\theta(p)\varphi(g))
=\Psi^{-1} \circ f_* \circ \Psi (\theta(p)\varphi(g))
=\theta(p')\varphi(h)\varphi(g) \in A[t^{\pm 1}].\]
We can now calculate~$f_*(at^\ell)$.
Pick~$g \in \pi_1(N) \ast \Z$ and~$p \in \pi_1(N)$ such that we have $\varphi(g)=t^{\ell}$ and~$\theta(p)=a$.
Now $f_*(at^\ell)=f_*(\theta(p)\varphi(g))=\theta(p')\varphi(h)t^\ell$, so the lemma follows by writing~$\varphi(h)=t^k$ and~$a':=\theta(p')$.
Then since~$[p']$ has the same order as~$[p]$, it follows that~$a'$ has the same order as~$a$.
\end{proof}

We can now prove the main result of this section.

\begin{proposition}
\label{thm:InfiniteAlgebra}
Fix an odd prime $q$.
For~$M=X_{2q}(U) \natural (S^1 \times D^3)$, the set\purple{s}
\begin{align*}
&\Aut (\Bl_{\partial M})/( \Aut(\lambda_M) \times \hAut_{\varphi}^+(\partial M) ) \text{ \purple{and} }\\
&\purple{\Aut (\Bl_{\partial M}, \mu_{\Bl_{\partial M}})/( \Aut(\lambda_M) \times \hAut_{\varphi}^{+,\quadr}(\partial M) )}
\end{align*}
\purple{are} countably infinite.
\end{proposition}

\begin{proof}
Fix identifications~$\pi_1(M)=\Z$ and~$(H_2(M;\Z[t^{\pm 1}]),\lambda_M)=(\Z[t^{\pm 1}],\lambda_{2q})$.  Lemma~\ref{lem:Rank1} implies that~$ \Aut (\Bl_{\partial M})/ \Aut(\lambda_M)=U(\Z[t^{\pm 1}]/2q)/U(\Z[t^{\pm 1}])$.
We know from Proposition~\ref{ex:UnitaryUnitsZZ} that~$U(\Z[t^{\pm 1}]/2q)/U(\Z[t^{\pm 1}]) \cong \Z$,  every element of which is of the form~$(q{-}1)t^n+q$ with~$n \in \Z$.
We will now show that there is a bijection of sets
$$ \Aut (\Bl_{\partial M})/( \Aut(\lambda_M) \times \hAut_{\varphi}^+(\partial M) ) \cong \Z.$$
In the notation of Proposition~\ref{prop:im(H_*)small}, we have~$N=L(2q,1)$ with~$\pi_1(L(2q,1)) \cong \Z/2q$ as well as~$A= H_1(\pi_1(N)) = \pi_1(N) = \Z/2q$.

Using Proposition~\ref{prop:im(H_*)small}, we will argue that any automorphism of the group~$H_1(\partial M;\Z[t^{\pm 1}]) \cong (\Z/2q)[t^{\pm 1}]$ induced by a homotopy equivalence~$f \in \Aut^+_\varphi(\partial M)$ is of the form~$p(t) \mapsto \pm t^k p(t)$, for some~$k \in \Z$.
To see this,  given~$p(t) \in (\Z/2q)[t^{\pm 1}]$,  by~$\Z[t^{\pm 1}]$-linearity of~$f_*$ we have  $f_*(p(t))=p(t)f_*(1)$.
By Proposition~\ref{prop:im(H_*)small}, $f_*(1) = a \cdot t^k$, for some $k \in \Z$ and some $a \in \Z/2q$.  We need to show that $a = \pm 1$.
Since $f_*$ is an isometry of $\Bl_{\partial M}$,  we also know that $a^2=1 \in \Z[t^{\pm 1}]/2q$; this holds because
\[\smfrac{-1}{2q} = \Bl_{\partial M}(1,1) = \Bl_{\partial M}(f_*(1),f_*(1)) = \Bl_{\partial M}(a\cdot t^k,a \cdot t^k) = \smfrac{-a^2}{2q} \in \Q(t)/\Z[t^{\pm 1}], \]
which implies that $a^2 =1 \in \Z[t^{\pm 1}]/2q$.  Then since $a \in \Z/2q$ we have that $a^2 =1 \in \Z/2q$.
Here we used that $\Bl_{\partial M} \cong -\partial \lambda_{2q}$ to compute the Blanchfield form~\cite[Proposition 3.5]{ConwayPowell}.

However the only elements of~$A=\Z/2q$ with $a^2=1$ are~$ \pm 1 \in \Z/2q$. Indeed such an $a$ belongs to $U(\Z/2q)\cong U(\Z/q) \times U(\Z/2)$. However $U(\Z/2)$ is trivial, so in fact $U(\Z/2q)\cong U(\Z/q)$.  We will show that $U(\Z/q)=\lbrace \pm 1\rbrace$. To see this, recall that for $q$ an odd prime the units $(\Z/q)^{\times}$ is a cyclic group of order $q-1$, and in such a group there is precisely one element of order 2.  Taken together with the trivial element there are therefore precisely two solutions to $x^2=1 \in (\Z/q)^{\times}$, namely $\pm 1$. So we see that $U(\Z/2q) \cong U(\Z/q)= \{\pm 1\}$.
It follows that $a = \pm 1$ and
$$f_*(p(t))=p(t)f_*(1)=\pm t^k p(t),$$
as asserted above.
In particular, observe that the action of a homotopy equivalence $f \in \hAut^+_\varphi(\partial M)$ is the same as the action by an element of $\Aut(\lambda_M) \cong U(\Z[t^{\pm 1}])$.
We deduce that
\begin{align*}
\Aut (\Bl_{\partial M})/( \Aut(\lambda_M) \times \hAut_{\varphi}^+(\partial M) ) \cong \Aut (\Bl_{\partial M})/\!\Aut(\lambda_M).
\end{align*}
But in Proposition~\ref{ex:UnitaryUnitsZZ} we computed the latter set to be
\[\Aut (\Bl_{\partial M})/\!\Aut(\lambda_M) \cong U(\Z[t^{\pm 1}]/2q)/U(\Z[t^{\pm 1}]) \cong \Z.\]
The inverse of these isomorphisms sends $n \in \Z$ to the automorphism given by multiplying by~$(q{-}1)t^n+q$.
%
\color{black}
Since the action of an element of $\hAut^+_\varphi(\partial M)$ is the same as the action by an element $\Aut(\lambda_M) \cong U(\Z[t^{\pm 1}])$, the same can be said for elements of $\hAut^{+,\quadr}_\varphi(\partial M) \subseteq \hAut^+_\varphi(\partial M)$.
As we mentioned in Remark~\ref{rem:DontHaveToWorkTooHard}, Lemma~\ref{lem:Quad=Sym} now implies that
\[ \Aut (\Bl_{\partial M}, \mu_{\Bl_{\partial M}})/( \Aut(\lambda_M) \times \hAut_{\varphi}^{+,\quadr}(\partial M) ) \cong  \Aut (\Bl_{\partial M})/( \Aut(\lambda_M) \times \hAut_{\varphi}^+(\partial M) ).\]
The second assertion in  Proposition~\ref{thm:InfiniteAlgebra} therefore follows from the first.
\color{black}
\end{proof}

%
%
%


\bibliographystyle{alpha}
\bibliography{BiblioInfinite}

\end{document}